\documentclass[12pt]{article}

\usepackage{amsmath}
\usepackage{amssymb}
\usepackage{pxfonts}
\usepackage{graphicx}
\usepackage{eucal}
\usepackage{mathrsfs}
\usepackage{theorem}
\usepackage{pifont}
\usepackage{color}
\usepackage[driverfallback=hypertex]{hyperref}
\usepackage[all]{xy}
\usepackage{abstract}

\newcommand{\brk}[1]{\left(#1\right)}          
\newcommand{\Brk}[1]{\left[#1\right]}          
\newcommand{\BRK}[1]{\left\{#1\right\}}        
\newcommand{\Abs}[1]{\left| #1 \right|}        
\newcommand{\Cases}[1]{\begin{cases} #1 \end{cases}}

\newcommand{\beq}{\begin{equation}}
\newcommand{\eeq}{\end{equation}}

\providecommand{\half}{\frac{1}{2}}

\newcommand{\calE}{\mathcal{E}}
\newcommand{\calM}{\mathcal{M}}
\newcommand{\calN}{\mathcal{N}}
\newcommand{\calS}{\mathcal{S}}
\newcommand{\calW}{\mathcal{W}}

\newcommand{\scrX}{\mathscr{X}}

\newcommand{\R}{\mathbb{R}}

\newcommand{\Textand}{\qquad\text{ and }\qquad}

\newcommand{\tr}{\operatorname{tr}}

\newcommand{\dist}{\operatorname{dist}}

\newcommand{\SFF}{\operatorname{II}}

\newcommand{\metric}{\mathfrak{g}}
\newcommand{\euc}{\mathfrak{e}}
\newcommand{\go}{{\metric}}
\newcommand{\ao}{{\mathfrak{a}}}
\newcommand{\bo}{{\mathfrak{b}}}
\newcommand{\co}{{\mathfrak{c}}}
\newcommand{\q}{{\mathfrak{q}}}
\newcommand{\qo}{{\q}}
\newcommand{\po}{{\mathfrak{p}}}
\newcommand{\tqo}{\tilde{\q}}
\newcommand{\hf}{\hat{f}}
\newcommand{\rr}{reduced}

\newcommand{\weakly}{{\rightharpoonup}}

\newcommand{\Volume}{d{\text{vol}_{\go}}}
\newcommand{\VolumeS}{d{\text{vol}_{\go|_\S}}}

\newcommand{\Volumet}{d{\text{vol}_{\tilde\go}}}

\newcommand{\W}{{\Omega}}
\renewcommand{\S}{{\calS}}
\newcommand{\NS}{{\calN\S}}
\newcommand{\SO}[1]{\text{$SO(#1)$}}
\newcommand{\GL}[1]{\text{$GL(#1)$}}

\newcommand{\id}{\operatorname{Id}}

\newcommand{\Ppar}{{P^\parallel}}
\newcommand{\Pparq}{P^\parallel_\qo}
\newcommand{\Pperpq}{P^\perp_\qo}
\newcommand{\Pperp}{{P^\perp}}
\newcommand{\ipar}{{\iota^\parallel}}
\newcommand{\iperp}{{\iota^\perp}}
\newcommand{\Dpar}{{\nabla^\parallel}}
\newcommand{\Dperp}{{\nabla^\perp}}
\newcommand{\qpar}{\qo^\parallel}
\newcommand{\qperp}{\qo^\perp}

\newcommand{\upi}{\underline{\pi_\star}}
\newcommand{\piPush}{\pi_\star}

\newcommand{\wpi}{\widetilde{\pi}}
\newcommand{\hpi}{\widehat{\pi}}
\newcommand{\hhpi}{\overline{\pi}}

\newcommand{\e}{\operatorname{e}}
\newcommand{\wexp}{\widetilde{\e}}
\newcommand{\hexp}{\widehat{\e}}
\newcommand{\hhexp}{\overline{\e}}

\newcommand{\wvarpi}{\widetilde{\varpi}}

\newcommand{\walpha}{\widetilde{\alpha}}

\newcommand{\Om}{\Xi}
\newcommand{\wOm}{\tilde{\Om}}

\def\Xint#1{\mathchoice
{\XXint\displaystyle\textstyle{#1}}%
{\XXint\textstyle\scriptstyle{#1}}%
{\XXint\scriptstyle\scriptscriptstyle{#1}}%
{\XXint\scriptscriptstyle\scriptscriptstyle{#1}}%
\!\int}
\def\XXint#1#2#3{{\setbox0=\hbox{$#1{#2#3}{\int}$ }
\vcenter{\hbox{$#2#3$ }}\kern-.6\wd0}}

\def\dashint{\Xint-}

\newcommand{\intW}{\int_{\W_h}}
\newcommand{\dashintW}{\dashint_{\W_h}}
\newcommand{\limH}{\lim_{h\to0}}

\newcommand{\liminfH}{\liminf_{h\to0}}

\newcommand{\Elim}{\calE_{\text{lim}}}
\newcommand{\Eh}{\calE_{h}}
\newcommand{\tEh}{\tilde{\calE}_{h}}

\newcommand{\Gh}{G_{h}}

\theoremheaderfont{\fontfamily{pzc}\bfseries\large}
\newtheorem{theorem}{Theorem}[section]
\newtheorem{lemma}{Lemma}[section]
\newtheorem{proposition}{Proposition}[section]

\newtheorem{corollary}{Corollary}[section]
\newtheorem{definition}{Definition}[section]
\newenvironment{proof}{{\flushleft \emph{Proof}:}}{\hfill\ding{110}}

\numberwithin{equation}{section}
\numberwithin{lemma}{section}
\numberwithin{definition}{section}
\numberwithin{theorem}{section}
\numberwithin{proposition}{section}
\numberwithin{corollary}{section}

\begin{document}

\title{A Riemannian Approach to Reduced Plate,  Shell, and Rod Theories}
\author{Raz Kupferman and Jake P. Solomon \\
\\
Institute of Mathematics, \\
The Hebrew University, \\
Jerusalem 91904 Israel}

\maketitle


\begin{abstract}
We derive a dimensionally-reduced limit theory for an $n$-dimensional nonlinear elastic body that is slender along $k$ dimensions.
The starting point is to view an elastic body as an $n$-dimensional Riemannian manifold together with a not necessarily isometric $W^{1,2}$-immersion in $n$-dimensional Euclidean space. The equilibrium configuration is the immersion that minimizes the average discrepancy between the induced and intrinsic metrics. The dimensionally reduced limit theory views the elastic body as a $k$-dimensional Riemannian manifold along with an isometric $W^{2,2}$-immersion in $n$-dimensional Euclidean space and linear data in the normal directions. The equilibrium configuration minimizes a functional depending on the average covariant derivatives of the linear data. The dimensionally-reduced limit is obtained using a $\Gamma$-convergence approach.  The limit includes as particular cases plate, shell, and rod theories. It applies equally to ``standard" elasticity and to ``incompatible" elasticity, thus including as particular cases so-called non-Euclidean plate, shell, and rod theories.
\end{abstract}

\tableofcontents

\section{Introduction}

The derivation of dimensionally-reduced elastic theories is a longstanding theme in material sciences, which goes back as far as to Euler, D. Bernoulli, Cauchy, and Kirchhoff \cite{Kir50}, and in the last century, to name just a few, to von Karman \cite{Kar10}, E. and F. Cosserat, Love \cite{Lov27}, and Koiter \cite{Koi66}. In essence, the problem is the following: Elasticity theory models the static and dynamic properties of three-dimensional bodies. Unless simplifying assumptions can be made, these models are highly nonlinear and notoriously complex. In many cases, however, one is interested in elastic bodies that are slender across one or two dimensions. In such cases it is appealing to model the body as an object of lower dimension---either a surface or a curve, depending on the number of slender axes. The challenge  is to derive reduced models for surface-like or curve-like bodies, departing from the full three-dimensional model.

This is the context to which belong \emph{plate}, \emph{shell}, and \emph{rod} theories \cite{TW59,Koi66}. Plates and shells are thin elastic sheets that are modeled as two-dimensional surfaces; plates are in a state of mechanical rest when they are planar, whereas shells are at rest in a non-planar configuration. Rods are thin and elongated bodies that are modeled as one-dimensional curves.

Elastostatics can be formulated using a variational approach, where the equilibrium configuration of the body is a minimizer of an energy functional defined on the space of configurations.
Until quite recently, the derivation of dimensionally-reduced models from three-dimensional elasticity had been based mostly on ansatzes for energy minimizers (e.g., the Kirchhoff-Love assumptions\cite{Kir50,Lov27}). Such an approach is non-rigorous, and in particular, different ansatzes lead to different reduced models (sometimes mutually inconsistent), a situation that has lead over the years to numerous controversies. This situation has changed drastically over the last decade with the development of new analytical methods, based on $\Gamma$-convergence, which have lead to rigorous derivations of plate \cite{FJM02b}, shell \cite{FJMM03}, and rod \cite{MM03} theories.

An underlying assumption of classical three-dimensional elastic theories is the existence of a  configuration in which the body is free of any internal stresses; this \emph{reference configuration} is unique modulo rigid transformations. Deviations from the reference configuration involve an elastic energy ``cost", and only occur in response to external forces or to the imposition of constraints, such as boundary conditions that are incompatible with the reference configuration. In the absence of such forces or constraints, the reference configuration is the equilibrium configuration.

There are many systems however, in which a stress-free configuration does not exist; such bodies are said to be \emph{residually-stressed}. The distinctive feature of a residually-stressed body is that its constituents change their shape if the body is dissected into parts (thus releasing the residual stress). The study and modeling of residually-stressed bodies has its own history, starting with the pioneering work of Bilby and co-workers \cite{BBS55,BS56} and Kondo \cite{Kon55}, followed about a decade later by
Wang \cite{Wan66,Wan67} and Kr\"oner \cite{Kro81}; most of the cited work addressed residual-stress in the context of defects and dislocations. In more recent years there has been a renewed interest in residually-stressed elastic bodies in the context of pattern formation in plants (see e.g. \cite{VG08,LM09,AESK11}) and in synthetic materials, such as thermo-responsive gels \cite{KES07}.

A residually-stressed body (made of an amorphous material) can be modeled as a three-dimensional Riemannian manifold $\calM$; the intrinsic property of the material is (local) distances between neighboring material elements.  Thus, the concept of a reference configuration is replaced by that of a \emph{reference metric} $\go$ \cite{Wan67}. A configuration of this body is a mapping $f:\calM\to\R^3$ from the Riemannian manifold into the ambient three-dimensional Euclidean space. The mapping $f$ induces on $\calM$ a metric $g = f_*\euc$, where $\euc$ is the canonical Euclidean metric in $\R^n$. The local energy density associated with the mapping $f$ is assumed to only depend on the metric discrepancy between $g$ and $\go$ (this metric discrepancy is called a \emph{strain} is the elastic context). Note that there exist also geometric treatments of residually-stressed materials in which extra-structure, in addition to the metric, is assumed; in this context see the work of Yavari and Goriely on geometric approaches to defects in solids \cite{YG12,YG12b,YG12c}.

If $\calM$ is simply-connected and the reference metric $\go$ is flat, namely, has a vanishing curvature tensor, then there exists an isometric immersion $f$ of $(\calM,\go)$ into $(\R^n,\euc)$, in which case $g=\go$, hence $f$ is a stress-free equilibrium configuration. If $\go$ is not flat, there does not exist a configuration for which $g=\go$, and the energy minimizing configuration will necessarily carry non-zero elastic energy, i.e., it will be residually-stressed. In the elastic context, the metric $\go$ is said to be \emph{incompatible}, and the elastic model is known as \emph{incompatible elasticity}.

Organisms such as plant organs, and the thermo-responsive gels studied in \cite{KES07} are  surface-like residually stressed materials. The interest in such systems has naturally lead to the development of dimensionally-reduced models in incompatible materials. A reduced theory of \emph{non-Euclidean plates} was developed in \cite{ESK08} based on a non-rigorous ansatz, followed in \cite{ESK09b} by a reduced theory of \emph{non-Euclidean shells} (we will elaborate on the distinction between the two cases in Section~\ref{sec:examples}).  Lewicka and Pakzad \cite{LP10} rigorously derived the limit of non-Euclidean plates using a $\Gamma$-convergence approach,  thus generalizing the result of Friesecke et al. \cite{FJM02b} (the latter being a particular case for a flat $\go$).

In this paper we present an analysis that generalizes in one fell swoop the derivation of (Euclidean) plate \cite{FJM02b}, shell \cite{FJMM03}, and rod \cite{MM03} theories, and (non-Euclidean) plate theory \cite{LP10}, and also provides, as particular cases, a rigorous derivation of non-Euclidean shell and rod theories, for which no current analyses exist. Specifically, we consider a family of elastic problems in which the domains $\W_h\subset\calM$ are a one-parameter family of $n$-dimensional Riemannian manifolds that are shrinking into an $(n-k)$-dimensional submanifold $\S$ as the thickness parameter $h$ tends to zero. For each such domain we consider immersions into an $n$-dimensional Euclidean space (in the physical context $n=3$ and $k=1,2$) and associate with each such immersion an energy $\Eh:W^{1,2}(\W_h;\R^n)\to\R$. Under suitable assumptions, we prove that any family of (possibly approximate) minimizers $f_h\in W^{1,2}(\W_h;\R^n)$ of $\Eh$ converges in a sense we define (modulo subsequences) as $h\to 0$ to a mapping $f\in W^{2,2}(\S;\R^n)$, which minimizes a limiting energy functional $\Elim:W^{2,2}(\S;\R^n)\to\R$. Thus, the elastic problem associated with the immersion of the $(n-k)$-dimensional manifold $\S$ into $\R^n$ is the dimensionally-reduced model in the $h\to0$ limit.

As stated above, this general result embodies the existing theories for plates, shells, and rods, the distinction between the first two cases being solely a property of the reference metric $\go$, whereas the distinction between plates/shells and rods is the codimension $k$. Indeed, an approach based on a reference metric rather than a reference configuration makes the distinction between plates and shells almost unnoticed. Moreover, under this viewpoint, there is nothing special about non-Euclidean plates, shells, or rods, except for the fact that the limiting energy functional $\Elim$ may not assume a state of zero energy. The remarkable fact is that the formulation of the problem within the framework of Riemannian geometry leads to a limiting model that has the exact same form in all instances.

While our result is very general in that it covers a variety of limits of elastic problems that were previously treated separately, we have deliberately restricted our attention to a specific energy density, which can be considered as a generalization of Hooke's law for isotropic materials, and to slender bodies with a symmetric cross section.  An extension of the present analysis to remove these restrictions seems to be  straightforward.

\section{Problem statement and main result}\label{sec:mr}

Let $\calM$ be an $n$-dimensional smooth oriented manifold; let $\S \subset \calM$ be an $m$-dimensional smooth oriented submanifold; let $k$ denote the codimension of $\S$ so that $m +k = n$. We endow $\calM$ with a smooth Riemannian metric $\go$. The submanifold $(\S,\go|_\S)$ is bounded, either closed, or having a Lipschitz continuous boundary. In either case, the Riemann curvature tensor is uniformly bounded in $\calM$. In the context of incompatible elasticity, the manifold $(\calM,\go)$ models an elastic body with  internal distances prescribed by the reference metric   $\go$.

Let $h\in(0,h_0)$ be a continuous parameter, and let
\[
\W_h = \{p\in\calM:\,\, \dist(p,\S)<h \} \subset\calM
\]
be a family of tubular neighborhoods of $\S$ that inherit the reference metric $\go$. The boundedness of the curvature implies that for small enough $h$, the exponential map,
\[
\exp: \{(p,\xi) \in \NS:\,\, |\xi|\le h\}  \to\W_h
\]
is a diffeomorphism, where $\NS = T\calM|_\S^\perp$  is the orthogonal complement of $T\calM|_\S^\parallel\cong T\S$ in $T\calM|_\S$. Thus, we identify
\[
T\calM|_\S \cong T\S \oplus \NS.
\]

As standard, we denote by $\pi:\NS \to\S$ the projection from the vector bundle
$\NS$, or its restriction to $\W_h$,
onto its base. Let $E \rightarrow S$ and $F \rightarrow \NS$ be vector bundles, and let $\Phi : \pi^* E \rightarrow F$. Let $\xi \in \NS$ and let $\eta \in E_{\pi(\xi)}$. The fiber $(\pi^* E)_\xi$ is canonically identified with the fiber $E_{\pi(\xi)}$. So, we can unambiguously apply $\Phi$ at $\xi$ to $\eta$. Denote the result by $\Phi_{\xi}(\eta)$.

We next introduce some definitions and notations.
We define the projection operators,
\[
\Ppar : T\calM|_\S \to T\S
\Textand
\Pperp : T\calM|_\S \to \NS,
\]
and the corresponding inclusions
\[
\ipar : T\S \hookrightarrow T\calM|_\S
\Textand
\iperp : \NS  \hookrightarrow T\calM|_\S.
\]
Let  $\nabla$ be the Levi-Civita connection on $T\calM|_\S$. Then, the induced connections on $T\S$ and $\NS$ are
\[
\Dpar = \Ppar\circ\nabla\circ\ipar
\Textand
\Dperp = \Pperp\circ\nabla\circ\iperp.
\]
When it does not cause confusion, we use $\nabla$ without any decorations to denote the induced connections as well.

Let $\iota : \pi^*\NS \hookrightarrow T\NS$ denote the canonical identification of the vector bundle $\NS$ with its own vertical tangent space. Specifically, for $\xi\in\NS$ and $\eta\in(\pi^*\NS)_{\xi}$, there is a canonical identification of $\eta$ with an element of $(\NS)_{\pi(\xi)}$. We then define a curve $\gamma:I\to\NS$,
\[
\gamma(t) = \xi + \eta \, t,
\]
and identify $\iota_\xi(\eta) = \dot{\gamma}(0)$.

As stated above, we identify $\W_h$ with a subset of $\NS$.
Consider now the tangent space $T\NS$. Define the isomorphism
\[
\Pi:\pi^* T\S \oplus \pi^*\NS \to T\NS,
\]
as follows. Let $\zeta$ denote the zero section of $\NS$. Define $\Pi$ to be the unique map such that
\[
\Pi|_{S} = d\zeta \oplus \iota|_S,
\]
and for each $\xi \in \NS$ we have
\[
\nabla_{\iota_\xi(\xi)} \Pi  = 0.
\]
That is, $\Pi$ is given by radial parallel transport along the fibers of $\NS$.

A notable property of $\Pi$ is that it preserves the metric, namely
\[
\go(\Pi_\xi u,\Pi_\xi v) = \go(u,v),
\qquad\forall \xi \in \NS,\quad u,v \in T_{\pi(\xi)}\calM.
\]

For every $h\in(0,h_0)$ we consider mappings $f_h\in W^{1,2}(\W_h;\R^n)$, and assign to every such mapping an energy $\Eh[f_h]$.
In (hyper-)elasticity the energy $\Eh$ is assumed to be a volume integral of a non-negative \emph{energy density} $W_h$ \cite{Tru55}. In the absence of external constraints and forces, $W_h$ only depends on the local value of the derivative of the mapping $df_h$, and only vanishes if $f_h$ is a local orientation-preserving isometry, namely, if $df_h\in\SO{n}$, where
\[
\SO{n} = \BRK{\qo:T\W_h\to\R^n:\,\, \qo^*\euc = \go,  \text{ $\qo$ is orientation-preserving}}.
\]
The mappings $f_h\in W^{1,2}(\W_h;\R^n)$ are only of interest modulo rigid transformations, hence we may assume (for the sake of a later compactness argument) that
\beq
\intW f_h\,\Volume = 0.
\label{eq:mean0}
\eeq

In the present work we consider a specific energy functional that postulates that the material is isotropic and that the energy density scales quadratically with the distance of $df_h$ from $\SO{n}$. Such an energy density can be viewed as a continuum variant of Hooke's law for linear springs (i.e., the energy density is quadratic in the local strain). Specifically, we define
\beq
\Eh[f_h] = \frac{1}{h^2} \dashintW \dist^2(df_h,\SO{n})\, \Volume,
\label{eq:Eh}
\eeq
where $\dashint$ denotes integration divided by the volume of the domain, and the additional $1/h^2$ prefactor is discussed next.

For every fixed $h\in(0,h_0)$ the energy functional \eqref{eq:Eh} defines an elastic problem: find the mapping $f_h\in W^{1,2}(\W_h;\R^n)$ that minimizes $\Eh$. It is not known a priori that such minimizers do exist, but one can always  consider a family $f_h$ of \emph{approximate minimizers}, defined by the condition
\[
\lim_{h\to0} \brk{\Eh[f_h] - \Eh^*} = 0,
\]
where
\[
\Eh^* =  \inf\BRK{\Eh[f]:\,\, f\in W^{1,2}(\W_h;\R^n)}.
\]

As $h\to0$ the family of $n$-dimensional domains $\W_h$ shrinks to the $m$-dimensional submanifold $\S$, and hence the volume integral of the energy density in the (exact or approximate) equilibrium configuration is expected to tend to zero. Since we are interested  in the $h\to0$ limit of this family of elastic problems, we have first rescaled the energy by dividing it by the volume of the domain. Furthermore, motivated by the physical setting in which $\S$ is a either a one- or two-dimensional submanifold of a three-dimensional manifold, one can expect $\W_h$ to be ``almost" $W^{1,2}$-isometrically immersible in $\R^n$, in the sense that even the energy per unit volume tends to zero as $h\to0$. This amounts to the  submanifold $(\S;\go|_\S)$ being isometrically immersible in $\R^n$ more regularly than $W^{1,2}$. As will be shown, if $(\S;\go|_\S)$ is $W^{2,2}$-isometrically immersible into $\R^n$, then the energy per unit volume is $O(h^2)$ as $h\to0$, which is why we divided the energy per unit volume in \eqref{eq:Eh} by the additional $1/h^2$ factor.

Our assumptions about the immersibility of $(\S,\go|_\S)$ in $\R^n$ are encapsulated in the so-called \emph{finite bending assumption}:

\begin{quote}
There exists a sequence of mappings $f_h\in W^{1,2}(\W_h;\R^n)$ such that
\beq
\Eh[f_h] = O(1).
\label{eq:fba}
\eeq
\end{quote}


We now state our main results:
Let
\[
\scrX = \BRK{(F,\qperp):\,\, F\in W^{2,2}(\S;\R^n), \qperp\in W^{1,2}(\S;\NS^*\otimes\R^n)}.
\]
We say that a sequence of maps $f_h : \Omega_h \rightarrow \R^n$ \emph{\rr-converges}
to an element $(F,\qperp) \in \scrX$ if
\beq
\limH\,\, \dashintW |f_h  - F\circ\pi|^2\,\Volume = 0,
\label{eq:we_show_1}
\eeq
and
\beq
\limH\,\, \dashintW |df_h\circ\Pi - \pi^*(dF\oplus\qperp)|^2\,\Volume = 0.
\label{eq:we_show_2}
\eeq
For the physically-oriented reader, condition \eqref{eq:we_show_1} states that the conformations $f_h$ of the shrinking domains $\W_h$ converge in the mean-square to a conformation $F$ of the submanifold $\S$. Condition \eqref{eq:we_show_2} states that the tangential component of $df_h$ consistently converges to the derivative of $F$, whereas the normal component of $df_h$ converges to a limit $\qperp$.

Here and throughout this paper, we denote by $\pi^* dF$ the section of $\pi^* T^*\S \otimes \R^n$ obtained by pulling-back $dF$ considered as a section of $T^*\S \otimes \R^n$. This should not be confused with the closely related pull-back of $dF$ considered as a $1$-form on $\S$ involving composition with $d\pi,$ which we denote by
\[
\pi^\star dF = \pi^*dF \circ d\pi = d\pi^* \circ \pi^*dF.
\]

The first step in our analysis is to show that any family of (possibly approximate) minimizers $f_h$ of $\Eh$ \rr-converges, modulo subsequences, to an element $(F,\qperp)$ of the space $\scrX$.

We introduce a functional $\Elim:\scrX\to\R\cup\{\infty\}$, defined as follows. For $(F,\qperp) \in \scrX$, define $\qo \in W^{1,2}(\S,T^*\calM|_S\otimes \R^n)$ by
\[
\qo = dF \oplus \qperp,
\]
and let $\Pparq\in\Gamma(\S;\R^n\otimes T\S)$ and $\Pperpq\in\Gamma(\S;\R^n\otimes\NS)$ be defined by
\[
\Pparq = \Ppar \circ \qo^{-1}
\Textand
\Pperpq = \Pperp \circ\qo^{-1}.
\]
Define
\beq
\Elim[F,\qperp] = \Cases{
\frac{\kappa}{2}\dashint_\S
\brk{2 |\Pparq \circ \nabla \qperp  - \SFF)|^2 + |\Pperpq\circ\nabla\qperp|^2} \, \VolumeS  & \qo \in\SO{n} \,\,\text{a.e.} \\
\infty &\text{otherwise},}
\label{eq:Elim}
\eeq
where $\go|_\S$ is the induced metric on $\S$ and $\kappa$ is the volume of the $k-1$ dimensional unit sphere divided by $k+2$.

We prove the following:
\begin{enumerate}
\item
The \emph{lower-semicontinuity property},
\[
\liminf_{h\to0} \Eh[f_h] \ge \Elim[F,\qperp].
\]

\item
For every $(\Phi,\po)\in\scrX$ there exists a family of mappings $\phi_h\in W^{1,2}(\W_h;\R^n)$ (a so-called \emph{recovery sequence}), such that $\phi_h$ \rr-converges to $(\Phi,\po)$ and
\[
\lim_{h\to0} \Eh[\phi_h] = \Elim[\Phi,\po].
\]
\end{enumerate}

It is easy then to show (see e.g., dal Maso \cite{Dal93} on $\Gamma$-convergence) that the (possibly partial) limit $(F,\qperp)$ of the sequence $f_h$ of (possibly approximate) minimizers is a (true!) minimizer of the limiting functional $\Elim$, and moreover that
\[
\Elim[F,\qperp] = \lim_{h\to0} \Eh[f_h].
\]

The practical implication of this result is the following: whenever faced with the need to find a minimizer $f_h:\W_h\to\R^n$ of $\Eh$ for sufficiently small  $h$, one can rather look for a minimizer $(F,\qperp)$ of $\Elim$, which approximates $f_h$ in the sense of \eqref{eq:we_show_1} and \eqref{eq:we_show_2}. In most cases, the latter task turns out to be easier.


\section{Preliminaries}
In this section, we collect a number of definitions, facts, and basic lemmas used throughout the paper.
\subsection{Integration}
Let $M$ be a manifold with boundary and let $N$ be a manifold without boundary.
Let $f : M \rightarrow N$ be smooth. Let $E \rightarrow N$ be a vector bundle. Denote by $A^l(N,E)$ differential forms of degree $l$ on $N$ with coefficients in $E$. Denote by $f^\star : A^*(N,E) \rightarrow A^*(M,f^*E)$ the pull-back of differential forms by $f$.
If $f$ and $f|_{\partial M}$ are proper submersions of relative dimensions $k$ and $k-1$ respectively, denote by $f_\star : A^*(M,f^*E) \rightarrow A^*(N,E)[-k]$ the push-forward operator or integration over the fiber of $f$. See Bott and Tu \cite{BT82} for a discussion of integration over the fiber in the case when $M$ is the total space of a vector bundle,
which is what we will use. The push-forward operator $f_\star$ satisfies the following properties:
\begin{enumerate}
\item\label{it:pt}
Let $N$ be the point so that $E, f^*E$, are trivial bundles. If $\alpha \in A^l(M,f^*E)$, then
\[
f_\star \alpha =
\begin{cases}
\int_M \alpha & l = \dim M, \\
0 & \text{otherwise.}
\end{cases}
\]
\item\label{it:m}
If $\beta \in A^*(M,f^*E)$ and $\alpha \in A^*(N,E)$, then
\beq
f_\star(f^\star\alpha\wedge\beta) = \alpha\wedge f_\star \beta.
\label{eq:push1}
\eeq
This is a generalization of the linearity of integration to the fibered context.
\item\label{it:fp}
Let
\[
\xymatrix{
P \ar[r]^g \ar[d]^h & M \ar[d]^f \\
Q \ar[r]^k & N
}
\]
be a commutative diagram of smooth maps, where $f$ is a proper submersion, $P$ is the fiber product $M \times_N Q$, and $g,h$ are the canonical projections. Then, $h$ is a proper submersion, and if $\alpha \in A^*(M,f^*E)$, then
\[
h_\star g^\star \alpha = k^\star f_\star \alpha.
\]
This is a generalization of the classical change of variables formula.
\newcounter{saveenumi}
\setcounter{saveenumi}{\value{enumi}}
\end{enumerate}
It is easy to see that properties~\eqref{it:pt}--\eqref{it:fp} uniquely characterize $f_\star$.
Moreover,
\begin{enumerate}
\setcounter{enumi}{\value{saveenumi}}
\item \label{it:comp}
Let
\[
P \stackrel{g}{\longrightarrow} M \stackrel{f}{\longrightarrow} N,
\]
where $g$ and $f$ are proper submersions. Then
\beq
f_\star \circ g_\star = (f\circ g)_\star.
\label{eq:push2}
\eeq
This is a generalization of Fubini's theorem.
\item
Let $\nabla$ denote a connection on $E$ as well as the associated pull-back connection on $f^*E$. Let $\omega \in A^l(M,f^*E)$. The following generalization of Stokes theorem holds:
\[
\nabla (f_\star \omega) = f_\star(\nabla\omega) + (-1)^{l + k} (f|_{\partial M})_\star(\omega).
\]
\end{enumerate}

The following special cases will be particularly useful. It follows from \eqref{eq:push1} that
if $\alpha = F$ is a zero-form on $\S$ and $\beta = \Volume$, then
\beq
\pi_\star(F\circ\pi\,\,\Volume) = F\,\pi_\star\Volume.
\label{eq:push1b}
\eeq

It follows from \eqref{eq:push2} and item \ref{it:pt} above that if
\[
\W_h \stackrel{\pi}{\longrightarrow} \S \stackrel{\psi}{\longrightarrow} \text{point},
\]
then for every differential form $\beta$ on $\NS$,
\beq
\int_\S \pi_\star (\beta) = \int_{\W_h} \beta,
\label{eq:push2b}
\eeq

\subsection{The tangent bundle of a tubular neighborhood}
In Section~\ref{sec:mr}, we defined an isomorphism $\Pi : \pi^* T\S \oplus \pi^* \NS  \rightarrow T\NS$ using the connection on $\calM$ and the identification of an open subset of $\NS$ with a tubular neighborhood of $\S$ in $\calM$ by the exponential map. In the present section, we will construct another isomorphism
\[
\sigma \oplus \iota : \pi^* T\S \oplus \pi^* \NS \rightarrow T\NS
\]
using only the induced connection on $\NS$. Lemma~\ref{lm:siPi} below estimates the discrepancy between $\sigma \oplus \iota$ and $\Pi$. Thus, we may take advantage of the linearity of $\sigma\oplus \iota$ to expedite calculations.

Let $\iota : \pi^*\NS \hookrightarrow T\NS$ denote the canonical identification of the vector bundle $\NS$ with its own vertical tangent space, as explained in Section~\ref{sec:mr}.
The differential $d\pi : T\NS\to \pi^*T\S$ is also defined canonically.   Clearly, \beq\label{eq:dpiio}
d\pi\circ\iota=0,
\eeq
which implies that
\[
\pi^*\NS \stackrel{\iota}{\hookrightarrow} T\NS \stackrel{d\pi}{\longrightarrow}  \pi^*T\S
\]
is a short exact sequence.

To fully determine an isomorphism $T\NS \cong \pi^*T\S \oplus \pi^*\NS$ we need a map
\[
\sigma : \pi^*T\S \to T\NS,
\]
such that
\beq\label{eq:dpisi}
d\pi\circ\sigma = \id.
\eeq
Here, we use the induced connection on $\NS$. Define $\sigma$ to be the unique map such that for any curve $\alpha : I \to S$ and any parallel normal field $\xi:I\to\NS$ along $\alpha$, we have
\[
\sigma_\xi(\dot \alpha) = \dot \xi.
\]
Thus, we have constructed an isomorphism,
\[
\sigma\oplus\iota : \pi^* T\S \oplus \pi^* \NS \to T\NS.
\]

Let $\lambda$ denote the tautological section of $\pi^* \NS$. That is, for $\xi \in \NS$,
\[
\lambda_\xi = \xi \in (\pi^*\NS)_\xi.
\]
Let $\SFF : \NS \otimes T\S \rightarrow T\S$ denote the second fundamental form.
Then:

\begin{lemma}\label{lm:siPi}
We have
\[
\sigma \oplus \iota - \Pi = \sigma \circ \pi^* \SFF \circ (\lambda \otimes \pi^*\Ppar) + O(\lambda^{\otimes 2}).
\]
\end{lemma}

\begin{proof}
We start by writing
\[
\sigma \oplus \iota - \Pi  = (\sigma - \Pi\circ\pi^*\ipar)\oplus(\iota - \Pi\circ\pi^*\iperp).
\]
First, we show that for $\xi\in\NS$
\beq \label{eq:ioh2}
|(\iota - \Pi \circ \pi^* \iperp)_\xi | = O(|\xi |^2).
\eeq
Indeed, let $p \in \S$ and $\zeta,\eta \in \NS_p$ be arbitrary. Consider the sections $\iota(\zeta),\iota(\eta)$ of $\pi^*\NS|_{\pi^{-1}(p)}$. Thinking of $p$ as a point in $\NS$, we claim that
\beq\label{eq:ii0}
\nabla_{\iota(\zeta)} \iota(\eta)|_p = 0.
\eeq
In fact, by symmetry of the connection,
\begin{align}
\nabla_{\iota(\zeta)} \iota(\eta)|_p &= \left.\frac{D}{dt}\left(\left.\frac{d}{ds} (t\zeta + s\eta) \right|_{s = 0} \right)\right|_{t = 0} \notag\\
&= \left.\frac{D}{ds}\left(\left.\frac{d}{dt} (t\zeta + s\eta) \right|_{t = 0} \right)\right|_{s = 0} \notag\\
&= \nabla_{\iota(\eta)} \iota(\zeta)|_p. \label{eq:iis}
\end{align}
Since the identification between $\Omega_h$ and a neighborhood of $\S \subset \NS$ is via the exponential map, we have
\beq\label{eq:iziz}
\nabla_{\iota(\zeta)} \iota(\zeta)|_p = 0
\eeq
for arbitrary $\zeta$. Equation~\eqref{eq:ii0} follows by the polarization identity from equations~\eqref{eq:iis} and~\eqref{eq:iziz}.

An immediate consequence of~\eqref{eq:ii0} is that
\[
|(\nabla_{\iota(\zeta)} \iota(\eta))_{s\xi}| = O(|s\xi|).
\]
So, by definition of $\Pi$,
\[
|\nabla_{\iota(\xi)} (\iota - \Pi\circ \iperp)_{s\xi} | = O(|s\xi|).
\]
Since
\[
\Pi \circ \iperp|_S = \iota|_S,
\]
equation~\eqref{eq:ioh2} follows by integrating radially.

It remains to show that
\beq \label{eq:sih2}
\sigma - \Pi \circ \pi^* \ipar  = \sigma \circ \pi^* \SFF \circ (\lambda \otimes \pi^*\Ppar) + O(\lambda^{\otimes 2}).
\eeq
Let $p \in \S$, let $\eta \in T_p \S$ and $\xi \in \NS_p$. Let $\alpha : I \rightarrow S$ with $\dot\alpha(0) = \eta$. Let $\beta$ be a parallel normal field along $\alpha$ with $\beta(0) = \xi$. For $s \in \R$, $s \beta$ is also a parallel normal field along $\alpha$. So, by definition of $\sigma$,
\[
\sigma_{s\xi}(\eta) = \left.\frac{d}{dt} s\beta \right|_{t = 0}.
\]
Applying the covariant derivative $\nabla$ to both sides, we calculate how $\sigma$ varies radially along $\xi:$
\begin{align}
\frac{D}{ds} \sigma_{s\xi}(\eta)|_{s = 0} &= \left. \frac{D}{ds}\left( \left.\frac{d}{dt} (s\beta) \right|_{t = 0}\right)\right|_{s=0}\label{eq:dds1}\\
& = \left.\frac{D}{dt} \left(\left.\frac{d}{ds} (s\beta)\right|_{s = 0} \right)\right|_{t = 0}\notag\\
& = \left.\frac{D}{dt} \iperp(\beta)\right|_{t = 0}\notag \\
& = \ipar \circ \SFF(\xi,\eta),\notag
\end{align}
where the last equality follows from the definition of the second fundamental form. On the other hand, by Leibniz's product rule,
\beq
\left. \frac{D}{ds} \sigma_{s\xi} \left( \SFF(s\xi,\eta)\right)\right|_{s= 0} = \ipar \circ \SFF(\xi,\eta). \label{eq:dds}
\eeq
Combining equations~\eqref{eq:dds1} and~\eqref{eq:dds} we obtain
\[
\frac{D}{ds} \sigma_{s\xi}(\eta)|_{s = 0} = \left. \frac{D}{ds} \sigma_{s\xi} \left( \SFF(s\xi,\eta)\right)\right|_{s= 0}.
\]
So, considering $p$ as a point of $\NS$, by definition of $\Pi,$
\[
\nabla_{\iota_p(\xi)} \left(\sigma - \Pi\circ \pi^*\ipar - \sigma \circ \pi^* \SFF \circ (\lambda \otimes \pi^*\Ppar)\right) = 0.
\]
Since
\[
\left.\left(\sigma - \Pi\circ \pi^*\ipar\right)\right|_{\S}  =  0  = \sigma \circ \pi^* \SFF \circ (\lambda \otimes \pi^*\Ppar)|_\S,
\]
equation~\eqref{eq:sih2} follows by integrating radially.
\end{proof}

\subsection{The metric on a tubular neighborhood}
\begin{corollary}\label{cy:siapg}
\[
\go \circ (\sigma\oplus\iota)^{\otimes 2} - \pi^*\go = O(\lambda).
\]
\end{corollary}
\begin{proof}
Since parallel transport preserves $\go$, we have
\[
\go \circ\Pi^{\otimes 2} = \pi^*\go.
\]
So, the corollary follows from Lemma~\ref{lm:siPi}.
\end{proof}

Let $\tilde \go$ denote the unique metric on the total space of $\NS$ such that
$\sigma \oplus \iota$ is an isometry. It's easy to see that $\tilde\go|_\S = \go|_\S$. The following Corollary follows immediately from the previous.
\begin{corollary}\label{cy:tgg}
\[
\tilde \go - \go = O(\lambda).
\]
\end{corollary}
In the rest of this paper we will write $\go$ instead of $\pi^*\go$ when it does not cause confusion.

\subsection{The volume form on a tubular neighborhood}
Let $E,F \rightarrow M$, be vector bundles and let $h: E \rightarrow F$ be a morphism of vector bundles. Denote by $\Lambda^a h : \Lambda^a E \rightarrow \Lambda^a F$ the associated vector bundle morphism between the $a^{th}$ exterior powers of $E$ and $F$.

Taking the determinant of $\Pi$, we have an isomorphism
\[
(\Lambda^n \Pi)^{-1*} : \Lambda^m \pi^* T^*\S \otimes \Lambda^k \pi^* \NS^* \rightarrow \Lambda^n T^*\NS.
\]

Write
\begin{align*}
\rho  &= (\sigma \oplus \iota)^{-1*} \circ (\pi^*\Ppar)^* : \pi^* T^*\S \rightarrow T^*\NS, \\
\theta  &= (\sigma \oplus \iota)^{-1*} \circ (\pi^*\Pperp)^* : \pi^* \NS^* \rightarrow T^* \NS.
\end{align*}
Note that
\begin{equation}\label{eq:tris}
\sigma^* \circ \rho = \id, \qquad \iota^* \circ \theta = \id, \qquad \sigma^* \circ \theta = 0, \qquad \iota^* \circ \rho = 0.
\end{equation}
Moreover, equations~\eqref{eq:tris} uniquely characterize $\rho$ and $\theta$.
Taking exterior powers, we have
\[
\Lambda^i \rho : \Lambda^i \pi^* T^*\S \rightarrow \Lambda^i T^*\NS, \qquad \Lambda^j \theta : \Lambda^j \pi^* \NS^* \rightarrow \Lambda^j T^* \NS.
\]
Moreover,
\begin{equation}\label{eq:wedge}
\bigoplus_{i+j = l} \Lambda^i \rho \wedge \Lambda^j\theta  = \Lambda^l (\sigma \oplus \iota)^{-1*} : \bigoplus_{i+j=l}\Lambda^i \pi^* T^*\S \otimes \Lambda^j \pi^* \NS^* \simeq \Lambda^l \pi^*(T^*\S \oplus \NS^*) \rightarrow \Lambda^l T^*\NS.
\end{equation}

Let $\tilde \eta$ be the unit norm section of $\Lambda^m T^*\S$ belonging to the orientation class, i.e. $\tilde \eta = \VolumeS$. Let $\tilde \omega$ be the unit norm section of $\Lambda^k \NS^*$ belonging to the orientation class determined by the orientations of $\calM$ and $\S$. Define
\[
\eta = \Lambda^m \rho \circ \pi^*\tilde \eta, \qquad \omega = \Lambda^k \theta \circ \pi^*\tilde \omega.
\]
In particular, $\eta \in A^m(\NS)$ and $\omega \in A^k(\NS)$. It is immediate from the definition that
\beq\label{eq:vtot}
  \eta \wedge\omega= \Volumet.
\eeq

\begin{corollary}\label{cy:h2}
\[
\eta \wedge \omega - \Volume = O(\lambda).
\]
\end{corollary}
\begin{proof}
The Corollary is an immediate consequence of equation~\eqref{eq:vtot} and Corollary~\ref{cy:siapg}.
\end{proof}

\begin{lemma}
\label{lm:pisrho}
Let $\alpha \in A^l(\S)$. Then
\[
\pi^\star \alpha = \Lambda^l\rho \circ \pi^* \alpha.
\]
\end{lemma}

\begin{proof}
By equations~\eqref{eq:dpiio} and~\eqref{eq:dpisi}, we have
\[
\rho = d\pi^*.
\]
So, the lemma follows from the definition of $\pi^\star$.
\end{proof}

\begin{lemma}
\label{lm:etvo}
\begin{equation*}
\eta  = \pi^\star \VolumeS.
\end{equation*}
\end{lemma}
\begin{proof}
Combine Lemma~\ref{lm:pisrho} for $l = m$ with the fact that
\[
\Lambda^m\rho \circ \pi^*\VolumeS = \Lambda^m\rho \circ \pi^*\tilde \eta = \eta.
\]
\end{proof}

\begin{lemma}
\label{lm:cm}
We have
\beq
\frac{\piPush \Volume}{|\W_h|} - \frac{\VolumeS}{|\S|} = O(h), \qquad |\S|\nu_k h^k - |\W_h| = O(h^{1+k}),
\label{eq:lim_measure}
\eeq
where $\nu_k$ is the volume of the $k$ dimensional unit ball.
\end{lemma}

\begin{proof}
By Corollary~\ref{cy:h2} and Lemma~\ref{lm:etvo}, we have
\[
\pi_\star (\omega) \VolumeS - \pi_\star \Volume = \pi_\star (\pi^\star \VolumeS\wedge\omega) - \pi_\star \Volume = O(h^{1+k}).
\]
Note that $\pi_\star(\omega)$ is the constant $\nu_k h^k$.
Integrating over $\S$, we have
\[
|\S|\nu_k h^k - |\W_h| = O(h^{1+k}).
\]
The lemma follows.
\end{proof}

\begin{lemma}\label{lm:om}
We have $\omega|_{\partial \W_h} = 0.$
\end{lemma}
\begin{proof}
Let $S_h\NS \subset \NS$ denote the radius $h$ sphere bundle inside of $\NS.$ Since we have identified $\W_h$ with a subset of $\NS$ via the exponential map, it follows that $\partial \W_h = S_h\NS.$
Let $\lambda^\perp \subset \pi^*\NS$ denote the rank $k-1$ subbundle that is the orthogonal complement of $\lambda$ in $\NS.$ By definition of $\iota,$
\[
\iota|_{{\lambda^\perp}|_{S_h\NS}} : \lambda^\perp|_{S_h\NS} \rightarrow T\partial \W_h.
\]
Moreover, since the connection $\nabla$ used to define $\sigma$ is metric,
\[
\sigma|_{{\pi^*T\S|_{S_h\NS}}} : \pi^*T\S|_{S_h\NS} \rightarrow T\partial \W_h.
\]
Counting dimensions, we conclude that
\[
\iota|_{{\lambda^\perp}|_{S_h\NS}} \oplus \sigma|_{{\pi^*T\S|_{S_h\NS}}} : \lambda^\perp|_{S_h\NS} \oplus \pi^*T\S|_{S_h\NS} \rightarrow T\partial \W_h
\]
is an isomorphism. So, it suffices to show that
\[
\Lambda^k(\iota|_{{\lambda^\perp}|_{S_h\NS}} \oplus \sigma|_{{\pi^*T\S|_{S_h\NS}}})^* \omega = 0.
\]
Indeed, by the definition of $\omega$ and the third of equations~\eqref{eq:tris},
\[
\Lambda^k(\iota|_{{\lambda^\perp}|_{S_h\NS}} \oplus \sigma|_{{\pi^*T\S|_{S_h\NS}}})^* \omega = \Lambda^k(\iota|_{{\lambda^\perp}|_{S_h\NS}})^* \circ \Lambda^k \theta \circ \pi^*\tilde \omega.
\]
But, $\Lambda^k(\iota|_{{\lambda^\perp}|_{S_h\NS}})^* = 0$ because $\lambda^\perp$ has rank $k-1.$ The lemma follows.
\end{proof}

\subsection{Rescaling a tubular neighborhood}
Define $\mu_h : \Omega_{h_0} \rightarrow \Omega_{h_0 h}$ by
\[
\mu_h(\xi) = h \xi.
\]
Clearly $\pi \circ \mu_h = \pi$. So, there is a canonical bundle map $\tilde \mu_h : \pi^* T\calM|_\S \rightarrow \pi^* T\calM|_\S$ covering $\mu_h$. By abuse of notation, we use $\tilde\mu_h$ to denote its own restriction to the summands $\pi^*T\S$ and $\pi^*\NS$ of $\pi^*T\calM|_S$.

\begin{lemma}
We have
\begin{align}
d\mu_h \circ \iota &= h (\iota \circ\tilde \mu_h) \qquad \label{eq:sc1} \\
d\mu_h\circ \sigma &=\sigma\circ \tilde\mu_h. \label{eq:sc2}
\end{align}
\end{lemma}

\begin{proof}
Let $\xi \in \Omega_{h_0}$, let $\zeta \in \NS_{\pi(\xi)}$ and let $\chi \in T_{\pi(\xi)}\S$. By definition $\iota_{\xi}(\zeta) = \dot \gamma(0)$, where $\gamma(t) = t\zeta + \xi$. So,
\[
d\mu_h \circ \iota_{\xi}(\zeta) = \left .\frac{d}{dt} \mu_h\circ \gamma(t) \right|_{t = 0}.
\]
Furthermore,
\[
\mu_h \circ \gamma(t) = h t \zeta + h\xi.
\]
On the other hand, $\iota_{h\xi}(\zeta) = \dot\delta(0)$ where $\delta(t) = t \zeta + h\xi$. Equation~\eqref{eq:sc1} follows from the fact that
\[
\left .\frac{d}{dt} \mu_h\circ \gamma(t) \right|_{t = 0} = h \dot\delta(0).
\]
Similarly, by definition $\sigma_{\xi}(\chi) = \dot\nu(0)$, where $\nu(t)$ is a parallel normal field along a path $\alpha(t)$ in $\S$ with $\dot\alpha(0) = \chi$ and $\nu(0) = \xi$. So,
\[
d\mu_h \circ \sigma_{\xi}(\chi) = \left . \frac{d}{dt} \mu_h \circ \nu(t) \right|_{t=0}.
\]
Moreover, $\mu_h \circ \nu(t) = h \nu(t)$.
On the other hand, $h\nu(t)$ is a parallel normal field along $\alpha(t)$ with $h\nu(0) = h\xi$. So,
\[
\sigma_{h\xi}(\chi) = \left . \frac{d}{dt} h \nu(t) \right|_{t=0},
\]
and equation~\eqref{eq:sc2} follows.
\end{proof}

\begin{corollary}
We have
\begin{gather}
\iota^* \circ d\mu_h^* = h (\tilde\mu_h^* \circ \iota^*), \qquad \sigma^* \circ d\mu_h^* = \tilde \mu_h^* \circ \sigma^*. \label{eq:sc3} \\
d\mu_h^* \circ \theta = h (\theta \circ \tilde \mu_h^*), \qquad  d\mu_h^* \circ \rho =  \rho\circ\tilde \mu_h^*.
\label{eq:sc4}
\end{gather}
\end{corollary}
\begin{proof}
Equations~\eqref{eq:sc3} are the duals of equations~\eqref{eq:sc1} and~\eqref{eq:sc2}. Equations~\eqref{eq:sc4} follow from equations~\eqref{eq:sc3} and the fact that equations~\eqref{eq:tris} characterize $\rho$ and $\theta$.
\end{proof}

\begin{corollary}\label{cy:scv}
We have
\[
\mu_h^\star \omega = h^k \omega, \qquad \mu_h^\star \eta = \eta.
\]
\end{corollary}

\begin{proof}
The corollary follows from the definition of $\omega$ and $\eta$ along with equations~\eqref{eq:sc4}.
\end{proof}

\begin{lemma}\label{lm:resc}
Let $f \in L^1(\Omega_{h_0h})$. Then
\[
\dashint_{\Omega_{h_0h}} f\, \Volume = \frac{1+ O(h)}{\nu_k h_0^k |\S|} \int_{\Omega_{h_0}} (f \circ \mu_h)\, \eta\wedge\omega.
\]
\end{lemma}
\begin{proof}
Using Corollary~\ref{cy:h2}, Lemma~\ref{lm:cm} and Corollary~\ref{cy:scv}, we calculate
\begin{align*}
\dashint_{\Omega_{h_0h}} f \, \Volume &= \frac{1+ O(h)}{\nu_k h_0^k h^k |\S|}\int_{\Omega_{h_0h}} f \,\eta\wedge\omega \\
&= \frac{1 + O(h)}{\nu_k h_0^k h^k|\S|} \int_{\Omega_{h_0}} (f \circ \mu_h) \, \mu_h^\star(\eta\wedge\omega) \\
& = \frac{1+ O(h)}{\nu_k h_0^k |\S|} \int_{\Omega_{h_0}} (f\circ \mu_h)\, \eta\wedge\omega.
\end{align*}
\end{proof}

Let
\[
*_S^i : \Lambda^i \pi^* T^*\S \rightarrow \Lambda^{m-i} \pi^* T^*\S, \qquad *_N^j : \Lambda^j \pi^*\NS^* \rightarrow \Lambda^{k-j} \NS^*,
\]
denote the Hodge star operators induced by the metric $\go$. Let
\[
\tilde *^l : \Lambda^l T^*\NS \rightarrow \Lambda^{n-l} T^*\NS
\]
denote the Hodge star operator induced by the metric $\tilde\go$.
Then,
\beq\label{eq:star}
\tilde *^l = \sum_{i+j = l}(\rho \circ *^i_S \circ \sigma^*) \otimes (\theta \circ *^j_N \circ \iota^*).
\eeq
Denote by $\mu_h^\star \tilde *$ the pull-back Hodge star operator, i.e.
\[
(\mu_h^\star \tilde *)(\mu_h^\star\alpha) = \mu_h^\star(\tilde *\alpha).
\]
So, $\mu_h^\star \tilde*$ is the Hodge star operator of the metric $\tilde \go$.
\begin{lemma}\label{lm:star}
\[
\mu_h^\star \tilde * = \tilde *.
\]
\end{lemma}
\begin{proof}
The lemma follows from equations~\eqref{eq:star},~\eqref{eq:sc3} and~\eqref{eq:sc4}.
\end{proof}

\subsection{Weak convergence on shrinking tubular neighborhoods}

\begin{definition}\label{df:wc}
A sequence of $L^2$ differential forms $\alpha_h \in A^l_{L^2}(\Omega_{hh_0})$ \emph{weakly converges} to zero if for all $\Phi \in A^{n-l}(\Omega_{h_0})$ we have
\[
\int_{\Omega_{h_0}} (\mu_h^\star \tilde\go)(\mu_h^\star\alpha_h,\Phi)\, \eta\wedge\omega \rightarrow 0
\]
as $h \to 0$. A sequence of sections $a_h \in L^2(\Omega_{hh_0},\pi^*\Lambda^l T^*\calM|_{\S})$ weakly converges to zero if the corresponding sequence $\alpha_h = \Lambda^l (\Pi^*)^{-1} \circ a_h \in A^l_{L^2}(\Omega_{hh_0})$ weakly converges to zero.
\end{definition}

\begin{lemma}\label{lm:wc}
A sequence $a_h \in L^2(\Omega_{hh_0},\pi^*\Lambda^l T^*\calM|_{\S})$ weakly converges to zero if and only if
\[
\dashint_{\Omega_{hh_0}} \go(a_h, \beta_h) \Volume \rightarrow 0
\]
for all sequences $\beta_h$ of the form
\[
\beta_h = h^{j-p}(\pi^*\beta \circ \lambda^{\otimes p})
\]
where $\beta \in L^2(\S;\Lambda^i T^*\S \otimes \Lambda^j \NS^*\otimes \NS^{*\otimes p})$ with $i+j = n-l$ and $p$ is arbitrary.
\end{lemma}
\begin{proof}
Let
\[
\tilde \beta_h = \left(\Lambda^i \rho \otimes \Lambda^j \theta\right) \circ \beta_h.
\]
In particular, $\tilde\beta_h$ is a family of $(n-l)$-forms on $\Omega_{h_0h}$. It is clear from the definition of $\lambda$ that $\lambda\circ \mu_h = h\lambda$. So, by equations~\eqref{eq:sc4} we have
\[
\mu_h^* \tilde \beta_h = \left(\Lambda^i \rho \otimes \Lambda^j \theta\right) \pi^* \beta \circ \lambda^{\otimes p}.
\]
The right-hand side of the preceding equation is clearly independent of $h$. So, we write
\[
\Phi_\beta = \mu_h^\star \tilde \beta_h.
\]
Let $\alpha_h = \Lambda^l (\Pi^*)^{-1} \circ a_h.$
By Lemma~\ref{lm:siPi}, Corollary~\ref{cy:siapg}, equation~\eqref{eq:wedge}, Corollary~\ref{cy:tgg} and Lemma~\ref{lm:resc},
\begin{align*}
\dashint_{\Omega_{hh_0}}\go(a_h ,\beta_h)\Volume & = \dashint_{\Omega_{hh_0}} \go(\Lambda^l(\sigma\oplus\iota)^* \circ \alpha_h, \beta_h)\Volume + O(h) \\
&=\dashint_{\Omega_{hh_0}} \go(\alpha_h,(\Lambda^i\rho \otimes \Lambda^j\theta)\circ\beta_h)\Volume + O(h) \\
&= (1 + O(h)) \dashint_{\Omega_{hh_0}} \tilde \go(\alpha_h,\tilde \beta_h)\Volume + O(h) \\
&= \frac{1 + O(h)}{\nu_k h_0^k |\S|}\int_{\Omega_{h_0}} (\mu_h^\star \tilde \go)(\mu_h^\star\alpha_h,\Phi_\beta)\eta\wedge\omega + O(h).
\end{align*}
Observe that $\Phi_\beta$ is an arbitrary $L^2$ form on $\Omega_{h_0}$ that is polynomial along the fibers of $\pi$. The lemma follows since polynomials are dense in $L^2$.
\end{proof}

\section{Rigidity}
\label{sec:rigidity}

The compactness property, whereby any sequence of approximate minimizers of $\Eh$ \rr-converges, is based on a rigidity theorem that can be viewed as a quantitative version of Liouville's theorem. A rigidity theorem for mappings $\R^n\to\R^n$ was proved by Friesecke et al. \cite{FJM02b}, paving the way to their derivation of plates, shell, and rod theories. In this section, we  present a generalization of the rigidity theorem of \cite{FJM02b} that applies to our Riemannian setting. Like Lewicka and Pakzad in \cite{LP10}, we base our proof on the theorem in Euclidean space. The notable difference between our formulation of the rigidity theorem and the ones in \cite{FJM02b} and \cite{LP10} is that in the Riemannian setting one has to adapt the notion of a spatially constant matrix. Another difference between the approach here and the above mentioned references is the use of a smoothing convolution operator rather than a partition of unity.

We introduce some more  notations. Consider the commutative diagram in Figure~\ref{fig:1}.
Recall that $\pi:\NS\to\S$ denotes the canonical projection. Moreover, let $\varpi:T\S\to\S$ denote the canonical projection, and let $\e:T\S\to\S$ denote the exponential map.  The other maps in the diagram are canonical projections of fiber products.

\begin{figure}
\begin{center}
\begin{xy}
(50,60)*+{\e^*\NS} = "v1";%
(50,40)*+{T\S} = "v2";%
(50,20)*+{\S} = "v3";%
(80,40)*+{\S} = "v4";%
(80,60)*+{\NS} = "v5";%
(20,60)*+{\hexp^*\e^*\NS} = "v6";%
(20,40)*+{\e^* T\S} = "v7";%
(20,20)*+{T\S} = "v8";%
(20,0)*+{\S} = "v9";%
(-10,0)*+{T\S} = "v10";%
(-10,20)*+{\varpi^*T\S} = "v11";%
%
{\ar@{->}_{\varpi} "v2"; "v3"};%
{\ar@{->}_{\e} "v2"; "v4"};%
{\ar@{->}_{\pi} "v5"; "v4"};%
{\ar@{->}_{\wexp} "v1"; "v5"};%
{\ar@{->}_{\wpi} "v1"; "v2"};%
{\ar@{->}_{\hhexp} "v6"; "v1"};%
{\ar@{->}_{\hexp} "v7"; "v2"};%
{\ar@{->}_{\e} "v8"; "v3"};%
{\ar@{->}_{\hpi} "v6"; "v7"};%
{\ar@{->}_{\wvarpi} "v7"; "v8"};%
{\ar@{->}_{\varpi} "v8"; "v9"};%
{\ar@{->}_{\varpi} "v10"; "v9"};%
{\ar@{->}_{\beta} "v11"; "v8"};%
{\ar@{->}_{\alpha} "v11"; "v10"};%
%
\end{xy}

\end{center}
\caption{}
\label{fig:1}
\end{figure}

Below, we will use repeatedly the following identities that follow from the commutativity of the diagram, and the properties of the push-forward operators:
\begin{subequations}
\begin{align}
& \int_\S \e_\star = \int_\S \varpi_\star \label{eq:sub1} \\
& \wpi_\star \wexp^* = \e^* \pi_\star \label{eq:sub2} \\
& \wvarpi_\star \hpi_\star \hhexp^* = \e^* \varpi_\star \wpi_\star    \label{eq:sub3} \\
& \hhexp^* \wpi^*  =  \hpi^* \hexp^* \label{eq:sub4} \\
& \e^*\varpi_\star = \wvarpi_\star \hexp^* \label{eq:sub5} \\
& \beta^*\varpi^* = \alpha^*\varpi^* \label{eq:sub6} \\
& \beta_\star \alpha^* = \varpi^*\varpi_\star \label{eq:sub7}.
\end{align}
\end{subequations}

For every $h\in(0,h_0)$ we define the indicator function $J_h:T\S\to\R$:
\[
J_h(p,\eta) = 1_{|\eta|< h}.
\]
We then define a family of non-negative test functions $K_h:T\S\to\R$ whose support is compactly embedded in the support of the $J_h$ and satisfying
\beq
\varpi_\star\brk{K_h\,\wpi_\star \wexp^*\Volume} =  \varpi_\star\brk{K_h\,\e^* \pi_\star \Volume} =  1,
\label{eq:Kh_norm}
\eeq
where the first equality follows from \eqref{eq:sub2}.
We further choose $K_h$ such that
\beq
C_1 h^n K_h \le J_h \le C_2 h^n K_{2n},
\label{eq:KhJh}
\eeq
and
\beq
C h^{n+1} |dK_h| \le J_h.
\label{eq:dKh}
\eeq

Next, we introduce the function $\Phi:T\S\to T\S$ defined by
\[
\e\circ\,\Phi = \varpi
\Textand
\varpi\circ\Phi = \e.
\]
Clearly, $\Phi$ is a diffeomorphism, and
\begin{subequations}
\begin{align}
& \Phi\circ\Phi = \id \label{eq:Phisub1} \\
& J_h\circ\Phi = J_h \label{eq:Phisub2}.
\end{align}
\end{subequations}
Changing variables, we may also express the normalization of $K_h$ as follows:
\beq
\e_\star\brk{(K_h\circ\Phi)\,\varpi^* \pi_\star \Volume} =  1.
\label{eq:norm_Kh}
\eeq


We then define the mapping $\Om:\e^* T\S\to \varpi^*T\S$:
\[
\Om(p,\eta,\zeta) = (p,\eta,\tau),
\]
where $\tau\in T_p\S$ is the unique vector satisfying:
\[
\exp_p(\tau) = \exp_{\exp_p(\eta)}(\zeta).
\]
We note that
\begin{subequations}
\begin{align}
& \beta\circ\Om = \wvarpi \label{eq:Wsub1} \\
& \e\circ\alpha\circ\Om = \e\circ\hexp \label{eq:Wsub2} \\
& \wvarpi_\star \Om^* = \beta_\star \label{eq:Wsub2b} \\
& \varpi^*\varpi_\star  = \beta_\star\alpha^*  \label{eq:Wsub2c},
\end{align}
\end{subequations}
that is,  the following diagram commutes:

\begin{xy}
(0,0)*+{T\S} = "v1";%
(-20,10)*+{\varpi^*T\S} = "v2";%
(10,10)*+{T\S} = "v3";%
(0,20)*+{\e^*T\S} = "v4";%
(20,20)*+{T\S} = "v5";%
(40,20)*+{\S} = "v6";%
{\ar@{->}_{\beta} "v2"; "v1"};%
{\ar@{->}_{\alpha} "v2"; "v3"};%
{\ar@{->}_{\Om} "v4"; "v2"};%
{\ar@{->}_{\wvarpi} "v4"; "v1"};%
{\ar@{->}_{\hexp} "v4"; "v5"};%
{\ar@{->}_{\e} "v5"; "v6"};%
{\ar@{->}_{\e} "v3"; "v6"};%
\end{xy}

We augment this diagram by adding the maps $\wOm$, $\walpha$ and $\hhpi$ as follows:

\begin{xy}
(0,0)*+{T\S} = "v1";%
(-20,10)*+{\varpi^*T\S} = "v2";%
(10,10)*+{T\S} = "v3";%
(0,20)*+{\e^*T\S} = "v4";%
(20,20)*+{T\S} = "v5";%
(40,20)*+{\S} = "v6";%
(-20,30)*+{\alpha^*\e^*\NS} = "v7";%
(10,30)*+{\e^*\NS} = "v8";%
(0,40)*+{\hexp^*\e^*\NS} = "v9";%
(20,40)*+{\e^*\NS} = "v10";%
(40,40)*+{\NS} = "v11";%
{\ar@{->}_{\beta} "v2"; "v1"};%
{\ar@{->}_{\alpha} "v2"; "v3"};%
{\ar@{->}_{\Om} "v4"; "v2"};%
{\ar@{->}_{\wvarpi} "v4"; "v1"};%
{\ar@{->}_{\hexp} "v4"; "v5"};%
{\ar@{->}_{\e} "v5"; "v6"};%
{\ar@{->}_{\e} "v3"; "v6"};%
{\ar@{->}_{\hhpi} "v7"; "v2"};%
{\ar@{->}_{\wpi} "v8"; "v3"};%
{\ar@{->}_{\walpha} "v7"; "v8"};%
{\ar@{->}_{\hpi} "v9"; "v4"};%
{\ar@{->}_{\wpi} "v10"; "v5"};%
{\ar@{->}_{\pi} "v11"; "v6"};%
{\ar@{->}_{\hhexp} "v9"; "v10"};%
{\ar@{->}_{\wexp} "v10"; "v11"};%
{\ar@{->}_{\wOm} "v9"; "v7"};%
{\ar@{->}_{\wexp} "v8"; "v11"};%
\end{xy}

hence
\begin{subequations}
\begin{align}
& \hhpi\circ\wOm = \Om\circ\hpi \label{eq:Wsub3} \\
& \wexp\circ\walpha\circ\wOm = \wexp\circ\hhexp \label{eq:Wsub4} \\
& \hpi_\star \wOm^* \walpha^* = \wOm^*\alpha^*\wpi_\star \label{eq:Wsub5}.
\end{align}
\end{subequations}


Let $\Psi:\varpi^* T\calM|_\S\to \e^* T\calM|_\S$ be the isomorphism given by parallel transport along geodesic rays; we can view $\Psi_{(p,\eta)}$ as a mapping from $T_p \calM|_\S$ to $T_{\exp_p(\eta)} \calM|_\S$.

\begin{lemma}
\beq
\Om^*\alpha^*\Psi - \hexp^*\Psi\circ \wvarpi^*\Psi = O(h^2).
\label{eq:Wh2}
\eeq
\end{lemma}

\begin{proof}
Holonomy around a loop is the integral of curvature over a spanning surface.
\end{proof}

We these preliminaries, we state a local rigidity theorem:

\begin{theorem}
\label{th:rig1}
There exists a constant $C>0$ such that for every $h\in(0,h_0)$ and every
$f_h\in W^{1,2}(\W_h;\R^n)$ there exists a section $\po_h\in L^2(\S;T^*\calM|_\S\otimes\R^n)$, such that
\[\begin{split}
& \varpi_\star \Brk{J_h\,\wpi_\star \brk{\Abs{\wexp^* df_h \circ \wexp^* \Pi\circ
\wpi^* \Psi - \wpi^* \varpi^* p_h}^2 \wexp^*\Volume}} \\
&\qquad
\le C \BRK{\varpi_\star \Brk{J_h \wpi_\star \wexp^*\brk{\dist^2(df_h,\SO{n})\,\Volume}}
+ h^{2} \,\varpi_\star \brk{J_h\, \wpi_\star \wexp^*\Volume}}.
\end{split}
\]
\end{theorem}

\begin{proof}
Fix $p\in \S.$ Let $B_hT_pS \subset T_pS$ denote the ball of radius $h,$ and let
\[
U_{p,h} = \{\xi \in \e^*\NS|_{B_hT_pS}:\,\, |\xi| \leq h\}.
\]
For sufficiently small $h$, we identify $U_{p,h}$ with the open subset of Euclidean space
\[
V_h = \left\{(x_1,\ldots,x_n) \in \R^n \left| \sum_{i = 1}^m x_i^2 \leq h^2, \; \sum_{j=1}^k x_{j + m}^2 \leq h^2\right.\right\}
\]
as follows. Let $\eta^i$ be a basis of $T_pS$ and let $\xi^j$ be a frame of $\e^*\NS|_{B_hT_pS}$ such that $\nabla \xi^j$ vanishes at $0 \in T_pS$ for all $j.$
The map $u_{p,h} : V_h \rightarrow U_{p,h}$ given by
\[
u_{p,h}(x) = \sum_{j = 1}^k x_{j + m} \xi^j \left(\sum_{i=1}^m x_i \eta^i\right)
\]
is a diffeomorphism and, therefore, defines coordinates on $U_{p,h}.$ We claim that with respect to the coordinates $x_i,$ the metric has the form
\beq\label{eq:goh2}
\go_{ij} = \delta_{ij} + O(h^2).
\eeq
We return to the proof of \eqref{eq:goh2} below.
Using this system of coordinates, we view maps $U_{p,h}\to\R^n$ as maps $\R^n\to\R^n$. By the rigidity theorem proved in \cite{FJM02b}, there exists a constant $C>0$ such that for every $f \in W^{1,2}(U_{p,h};\R^n)$ there exists an $n\times n$ matrix $Q\in\SO{n}$, such that
\[
\dashint_{U_{p,h}} |\nabla f - Q|^2\, dx \le C \dashint_{U_{p,h}} \dist^2(\nabla f,\SO{n})\, dx,
\]
where all inner products here are  Euclidean. The theorem follows by using the fact that the Euclidean metric differs from the Riemannian metric by an $O(h^2)$ term.

It remains to prove~\eqref{eq:goh2}. For the rest of this proof, $\exp$ denotes the exponential map of $\calM$ while $\e$ denotes the exponential map of $S.$ Let $f_{p,h} : U_{p,h} \rightarrow \W_h$ be given by
\[
f_{p,h}(\xi) = \exp_{\e_p(\wpi(\xi))}(\xi).
\]
The metric on $U_{p,h}$ relevant to formula~\eqref{eq:goh2} is $f_{p,h}^* \go.$
Let $\xi,\chi$ be sections of $\e^*\NS$ with
\beq\label{eq:dxidchi}
\nabla\xi = \nabla \chi = 0
\eeq
at $0 \in T_pS.$ Let $\eta,\nu \in T_pS$, and let
\[
g(s,t) = f_{p,h}([t(\xi + s \chi)]_{t(\eta + s \nu)}).
\]
Denote by $Q$ the vector field along $g(0,t)$ given by
\[
Q(t) = \frac{\partial g}{\partial s}(0,t).
\]
To prove~\eqref{eq:goh2}, it suffices to show that
\[
|Q(t)|^2 = t^2 (|\chi(0)|^2 + |\nu|^2) + O(t^4).
\]
Indeed, it is easy to see that $Q(0) = 0.$ We denote the  covariant derivatives of $Q$ by $Q',Q''$ etc. Using the symmetry of the connection,
\begin{align*}
Q'(0) &= \left.\left.\frac{D}{\partial t}\frac{\partial g}{\partial s}\right|_{s=0}\right|_{t = 0} \\
&= \left.\left.\frac{D}{\partial s}\frac{\partial g}{\partial t}\right|_{t=0}\right|_{s = 0} \\
&= \left.\frac{\partial}{\partial s}(\xi(0) + s\chi(0) + \eta + s \nu)\right|_{s = 0}\\
&= \chi(0) + \nu.
\end{align*}
Let $R$ denote the curvature of $\go.$ Using the symmetry of the connection and assumption~\eqref{eq:dxidchi},
\begin{align*}
Q''(t) & = \left .\frac{D}{\partial t}\left .\frac{D}{\partial s} \frac{\partial g}{\partial t} \right|_{s = 0} \right|_{t = 0} \\
& = R\left(\frac{\partial g}{\partial t}(0,0),Q(0)\right)\frac{\partial g}{\partial t}(0,0) + \left .\frac{D}{\partial s} \left .\frac{D}{\partial t} \frac{\partial g}{\partial t} \right |_{t=0} \right |_{s = 0} \\
& = 0 + \left .\frac{D}{\partial s} \nabla_{\eta + s\nu} (\xi + s \chi) \right |_{s = 0}  = 0.
\end{align*}
So,
\begin{align*}
\langle Q,Q \rangle'(0) &= 2\langle Q',Q \rangle(0) = 0, \\
\langle Q,Q \rangle''(0) &= 2\langle Q',Q'\rangle(0) + 2\langle Q'',Q\rangle(0) = 2 |\chi(0) + \nu|^2, \\
\langle Q,Q\rangle'''(0) &= 6 \langle Q',Q''\rangle(0) + 2\langle Q''',Q\rangle(0) = 0,
\end{align*}
and \eqref{eq:goh2} follows.
\end{proof}

This local rigidity theorem is the basis for proving the following ``global" rigidity theorem:

\begin{theorem}
\label{th:rig2}
There exists a constant $C>0$ such that for every $h\in(0,h_0)$ and every
$f_h\in W^{1,2}(\W_h;\R^n)$ there exists a section
$\qo_h\in W^{1,2}(\S;T^*\calM|_\S \otimes \R^n)$,
such that
\beq
\dashint_{\W_h} |df_h\circ\Pi - \pi^* \qo_h|^2\,\Volume \le
Ch^2 \BRK{\Eh[f_h]  + 1}.
\label{eq:rig1}
\eeq
and
\beq
\dashint_\S |\nabla \qo_h|^2\,\frac{\piPush\Volume}{|\W_h|/|\S|} \le
C \BRK{\Eh[f] + 1}.
\label{eq:rig2}
\eeq
\end{theorem}

\begin{proof}
Let $\po_h\in L^2(\S;T^*\calM|_\S \otimes \R^n)$ be a section satisfying the assertion of Theorem~\ref{th:rig1}. We define
\beq
\qo_h = \varpi_\star\brk{K_h\,\wpi_\star\brk{\wexp^* df_h\circ\wexp^*\Pi\circ \wpi^*\Psi\otimes\wexp^*\Volume}}.
\label{eq:qh}
\eeq

Consider now the integral
\[
I = \intW |df_h\circ\Pi - \pi^* \qo_h|^2\,\Volume.
\]
Then.
\[
\begin{split}
I &= \int_S \pi_\star\brk{|df_h\circ\Pi - \pi^* \qo_h|^2\,\Volume} \\
&\stackrel{\eqref{eq:norm_Kh}}{=} \int_S \pi_\star\brk{|df_h\circ\Pi - \pi^* \qo_h|^2\,\Volume} \e_\star\brk{(K_h\circ\Phi)\,\varpi^* \pi_\star \Volume} \\
&\stackrel{\eqref{eq:push1}}{=} \int_S \e_\star\BRK{(K_h\circ\Phi)\,\varpi^* \pi_\star \Volume \wedge \e^*\pi_\star\brk{|df_h\circ\Pi - \pi^* \qo_h|^2\,\Volume}} \\
&\stackrel{\eqref{eq:sub1},\eqref{eq:sub2}}{=}  \int_S \varpi_\star\BRK{(K_h\circ\Phi)\,\varpi^* \pi_\star \Volume \wedge  \wpi_\star\wexp^*\brk{|df_h\circ\Pi - \pi^* \qo_h|^2\,\Volume}} \\
&\stackrel{\eqref{eq:push1}}{=} \int_S  (\pi_\star \Volume)\,  \varpi_\star\BRK{(K_h\circ\Phi)\, \wpi_\star\brk{|\wexp^*df_h\circ\wexp^*\Pi - \wpi^*\e^* \qo_h|^2\,\wexp^*\Volume}},
\end{split}
\]
where here and in the remainder of the proof we write above the relation signs the equation number from which the relation follows.

Using next the fact that parallel transport is norm-preserving, adding and subtracting $\wpi^*\varpi^*\po_h$, and using the Cauchy-Schwarz inequality, we get
\[
I \le 2 I_1 + 2 I_2,
\]
where
\[
\begin{aligned}
I_1 &= \int_S  (\pi_\star \Volume)\,  \varpi_\star\BRK{(K_h\circ\Phi)\, \wpi_\star\brk{|\wexp^*df_h\circ\wexp^*\Pi\circ\wpi^*\Psi - \wpi^*\varpi^*\po_h|^2\,\wexp^*\Volume}} \\
I_2 &= \int_S  (\pi_\star \Volume)\,  \varpi_\star\BRK{(K_h\circ\Phi)\, \wpi_\star\brk{|\wpi^*\varpi^*\po_h - \wpi^*\e^* \qo_h\circ\wpi^*\Psi|^2\,\wexp^*\Volume}}.
\end{aligned}
\]

Consider $I_1$:
\[
\begin{split}
I_1 &\stackrel{\eqref{eq:KhJh},\eqref{eq:Phisub2}}{\le}  \frac{C}{h^n} \int_S  (\pi_\star \Volume)\,  \varpi_\star\BRK{J_h\, \wpi_\star\brk{|\wexp^*df_h\circ\wexp^*\Pi\circ\wpi^*\Psi - \wpi^*\varpi^*\po_h|^2\,\wexp^*\Volume}} \\
&\stackrel{\eqref{eq:rig1}}{\le} \frac{C}{h^n} \int_S  (\pi_\star \Volume)\,  \varpi_\star\BRK{J_h\, \wpi_\star\brk{\dist^2(df_h,\SO{n})\circ\wexp(\wexp^*\Volume)}} \\
&\quad + \frac{C h^2}{h^n} \int_S  (\pi_\star \Volume)\,  \varpi_\star\BRK{J_h\, \wpi_\star\wexp^*\Volume} \\
&\stackrel{\eqref{eq:KhJh}}{\le} C \int_S  (\pi_\star \Volume)\,  \varpi_\star\BRK{K_{2h}\, \wpi_\star\brk{\dist^2(df_h,\SO{n})\circ\wexp\,(\wexp^*\Volume)}} \\
&\quad + C h^2 \int_S  (\pi_\star \Volume)\,  \varpi_\star\BRK{K_{2h}\, \wpi_\star\wexp^*\Volume}
\equiv I_{1a} + I_{1b}
\end{split}
\]

We then basically revert the steps we did before:
\[
\begin{split}
I_{1a} &\stackrel{\eqref{eq:push1}}{\le} C \int_S  (\pi_\star \Volume)\,  \varpi_\star\BRK{K_{2h}\, \wpi_\star\wexp^*\brk{\dist^2(df_h,\SO{n})\,\Volume}} \\
&\stackrel{\eqref{eq:sub2}}{=} C \int_S  (\pi_\star \Volume)\,  \varpi_\star\BRK{K_{2h}\, \e^*\pi_\star\brk{\dist^2(df_h,\SO{n})\,\Volume}} \\
&\stackrel{\eqref{eq:push1}}{=} C \int_S  \varpi_\star\BRK{K_{2h}\, \e^*\pi_\star\brk{\dist^2(df_h,\SO{n})\,\Volume} \wedge \varpi^*\pi_\star \Volume} \\
&\stackrel{\eqref{eq:sub1}}{=} C \int_S  \e_\star\BRK{K_{2h}\, \e^*\pi_\star\brk{\dist^2(df_h,\SO{n})\,\Volume} \wedge \varpi^*\pi_\star \Volume} \\
&\stackrel{\eqref{eq:push1}}{=} C \int_S  \pi_\star\brk{\dist^2(df_h,\SO{n})\,\Volume}  \e_\star\BRK{K_{2h}\, \varpi^*\pi_\star \Volume} \\
&\stackrel{\eqref{eq:Kh_norm}}{\le} C \int_S  \pi_\star\brk{\dist^2(df_h,\SO{n})\,\Volume}
\stackrel{\text{def}}{=} C h^2 |\W_h|\,\Eh[f_h],
\end{split}
\]
and
\[
\begin{split}
I_{1b} &=  C h^2 \int_S  \pi_\star \Volume\,  \varpi_\star\BRK{K_{2h}\, \wpi_\star\wexp^*\Volume} \\
&\stackrel{\eqref{eq:sub2}}{=} C h^2 \int_S  \pi_\star \Volume\,  \varpi_\star\BRK{K_{2h}\, \e^* \pi_\star\Volume} \\
&\stackrel{\eqref{eq:Kh_norm}}{\le} C h^2 \int_S  \pi_\star \Volume = Ch^2\,|\W_h|.
\end{split}
\]

Consider next $I_2$:
\beq
\begin{split}
I_2
&\stackrel{\eqref{eq:push1}}{=} \int_S  (\pi_\star \Volume)\,  \varpi_\star\BRK{(K_h\circ\Phi) \, |\varpi^*\po_h - \e^* \qo_h\circ\Psi |^2 \,
\wpi_\star\wexp^*\Volume}.
\end{split}
\label{eq:I2}
\eeq
We derive the following identity:
\[
\begin{split}
\e^* \qo_h\circ\Psi &\stackrel{\eqref{eq:qh}}{=} \e^* \varpi_\star \BRK{K_h\,\wpi_\star\brk{(\wexp^*df_h\circ\wexp^*\Pi\circ\wpi^*\Psi)\,\wexp^*\Volume}}\circ\Psi \\
&\stackrel{\eqref{eq:push1}}{=} \e^* \varpi_\star \wpi_\star\brk{(K_h\circ\wpi)(\wexp^*df_h\circ\wexp^*\Pi\circ\wpi^*\Psi)\,\wexp^*\Volume}\circ\Psi \\
&\stackrel{\eqref{eq:sub3}}{=}  \wvarpi_\star \hpi_\star \hhexp^*\brk{(K_h\circ\wpi)(\wexp^*df_h\circ\wexp^*\Pi\circ\wpi^*\Psi)\,\wexp^*\Volume}\circ\Psi \\
&\stackrel{}{=}  \wvarpi_\star \hpi_\star \brk{(K_h\circ\wpi\circ\hhexp)(\hhexp^*\wexp^*df_h\circ\hhexp^*\wexp^*\Pi\circ\hhexp^*\wpi^*\Psi)\,\hhexp^*\wexp^*\Volume}\circ\Psi \\
&\stackrel{\eqref{eq:push1}}{=} \wvarpi_\star \hpi_\star \brk{(K_h\circ\wpi\circ\hhexp)(\hhexp^*\wexp^*df_h\circ\hhexp^*\wexp^*\Pi\circ\hhexp^*\wpi^*\Psi\circ\hpi^*\wvarpi^*\Psi )\,\hhexp^*\wexp^*\Volume}\\
&\stackrel{\eqref{eq:sub4},\eqref{eq:Wsub4}}{=}
\wvarpi_\star \hpi_\star \brk{(K_h\circ\hexp\circ\hpi)(\wOm^*\walpha^*\wexp^*df_h\circ\wOm^*\walpha^*\wexp^*\Pi\circ\hpi^*\hexp^*\Psi\circ\hpi^*\wvarpi^*\Psi )\,\wOm^*\walpha^*\wexp^*\Volume}\\
&\stackrel{\eqref{eq:Wh2}}{=}
\wvarpi_\star \hpi_\star \brk{(K_h\circ\hexp\circ\hpi)(\wOm^*\walpha^*\wexp^*df_h\circ\wOm^*\walpha^*\wexp^*\Pi\circ\hpi^*\Om^*\alpha^*\Psi )\,\wOm^*\walpha^*\wexp^*\Volume} + O(h^2)\\
&\stackrel{\eqref{eq:push1}}{=}
\wvarpi_\star\BRK{(K_h\circ\hexp) \hpi_\star \wOm^*\walpha^*\brk{(\wexp^*df_h\circ\wexp^*\Pi)\,\wexp^*\Volume}\circ\Om^*\alpha^*\Psi }+ O(h^2) \\
&\stackrel{\eqref{eq:Wsub5}}{=}
\wvarpi_\star\BRK{(K_h\circ\hexp) \wOm^*\alpha^*\wpi_\star\brk{(\wexp^*df_h\circ\wexp^*\Pi)\,\wexp^*\Volume}\circ\Om^*\alpha^*\Psi }+ O(h^2) \\
&\stackrel{\eqref{eq:push1}}{=}
\wvarpi_\star\BRK{(K_h\circ\hexp) \Om^*\alpha^*\wpi_\star\brk{(\wexp^*df_h\circ\wexp^*\Pi\circ\wpi^*\Psi)\,\wexp^*\Volume} }+ O(h^2) \\
&\stackrel{}{=}
\wvarpi_\star\Om^*\BRK{(K_h\circ\hexp\circ\Om^{-1}) \alpha^*\wpi_\star\brk{(\wexp^*df_h\circ\wexp^*\Pi\circ\wpi^*\Psi)\,\wexp^*\Volume} }+ O(h^2) \\
&\stackrel{\eqref{eq:Wsub2b}}{=}
\beta_\star\BRK{(K_h\circ\hexp\circ\Om^{-1}) \alpha^*\wpi_\star\brk{(\wexp^*df_h\circ\wexp^*\Pi\circ\wpi^*\Psi)\,\wexp^*\Volume} }+ O(h^2).
\end{split}
\]
On the other hand, using the fact that
\[
\begin{split}
\e^* 1 &= \e^* \varpi^*\brk{K_h \wpi_\star \wexp^* \Volume} \\
&\stackrel{\eqref{eq:sub2},\eqref{eq:sub5}}{=}  \wvarpi_\star \hexp^* \brk{K_h \e^*\pi_\star \Volume} \\
&\stackrel{}{=}  \wvarpi_\star  \brk{(K_h\circ\hexp) \hexp^* \e^* \pi_\star \Volume} \\
&\stackrel{\eqref{eq:Wsub2}}{=}  \wvarpi_\star  \brk{(K_h\circ\hexp) \Om^*\alpha^* \e^* \pi_\star \Volume} \\
&\stackrel{\eqref{eq:push1},\eqref{eq:sub2}}{=}  \wvarpi_\star  \Om^* \brk{(K_h\circ\hexp\circ\Om^{-1}) \alpha^* \wpi_\star \wexp^*\Volume} \\
&\stackrel{\eqref{eq:Wsub2b}}{=}  \beta_\star  \brk{(K_h\circ\hexp\circ\Om^{-1}) \alpha^* \wpi_\star \wexp^*\Volume} \\
\end{split}
\]
we have
\[
\begin{split}
\varpi^*\po_h &= \varpi^*\po_h\, \beta_\star  \brk{(K_h\circ\hexp\circ\Om^{-1}) \alpha^* \wpi_\star \wexp^*\Volume} \\
&\stackrel{\eqref{eq:push1}}{=} \beta_\star  \brk{(K_h\circ\hexp\circ\Om^{-1}) (\beta^*\varpi^*\po_h) \alpha^* \wpi_\star \wexp^*\Volume} \\
&\stackrel{\eqref{eq:sub6}}{=} \beta_\star  \brk{(K_h\circ\hexp\circ\Om^{-1}) (\alpha^*\varpi^*\po_h) \alpha^* \wpi_\star \wexp^*\Volume} \\
&\stackrel{\eqref{eq:sub6}}{=} \beta_\star  \brk{(K_h\circ\hexp\circ\Om^{-1}) \alpha^* \wpi_\star \brk{\wpi^*\varpi^*\po_h \,\wexp^*\Volume}}.
\end{split}
\]
Thus,
\[
\begin{split}
& \e^* \qo_h\circ\Psi - \varpi^*\po_h = \\
&\quad \beta_\star\BRK{(K_h\circ\hexp\circ\Om^{-1}) \alpha^*\wpi_\star\brk{(\wexp^*df_h\circ\wexp^*\Pi\circ\wpi^*\Psi - \wpi^*\varpi^*\po_h)\,\wexp^*\Volume} }+ O(h^2),
\end{split}
\]
and by Jensen's inequality,
\[
\begin{split}
& |\e^* \qo_h\circ\Psi - \varpi^*\po_h|^2 \le \\
&\quad \beta_\star\BRK{(K_h\circ\hexp\circ\Om^{-1}) \alpha^*\wpi_\star\brk{|\wexp^*df_h\circ\wexp^*\Pi\circ\wpi^*\Psi - \wpi^*\varpi^*\po_h|^2\,\wexp^*\Volume} } \\
&\quad \stackrel{\eqref{eq:push1}}{=} \beta_\star\alpha^* \BRK{(K_h\circ\hexp\circ\Om^{-1}\circ\alpha) \wpi_\star\brk{|\wexp^*df_h\circ\wexp^*\Pi\circ\wpi^*\Psi - \wpi^*\varpi^*\po_h|^2\,\wexp^*\Volume} } \\
&\quad \stackrel{\eqref{eq:sub7}}{=}\varpi^*\varpi_\star \BRK{(K_h\circ\hexp\circ\Om^{-1}\circ\alpha) \wpi_\star\brk{|\wexp^*df_h\circ\wexp^*\Pi\circ\wpi^*\Psi - \wpi^*\varpi^*\po_h|^2\,\wexp^*\Volume} }.
\end{split}
\]
We then use the fact that $K_h\circ\hexp\circ\Om^{-1}\circ\alpha \le C/h^n\, J_{2h}$ and the local rigidity theorem to obtain
\[
\begin{split}
|\e^* \qo_h\circ\Psi - \varpi^*\po_h|^2 &\le
\frac{C}{h^{n}} \varpi^*\varpi_\star \BRK{J_{2h} \wpi_\star\brk{\dist^2(df_h,\SO{n})\circ\wexp\,\, \wexp^*\Volume} } + O(h^2).
\end{split}
\]

Substituting into \eqref{eq:I2}:
\[
\begin{split}
I_2  &\le \frac{C}{h^{n}}  \int_S  (\pi_\star \Volume)\,  \varpi_\star\BRK{(K_h\circ\Phi) \,\varpi^*\varpi_\star \BRK{J_{2h} \wpi_\star\brk{\dist^2(df_h,\SO{n})\circ\wexp\,\, \wexp^*\Volume} } \,
\wpi_\star\wexp^*\Volume} \\
&\quad \stackrel{\eqref{eq:push1}}{=} \frac{C}{h^{n}}  \int_S  (\pi_\star \Volume)\,
\varpi_\star \BRK{J_{2h} \wpi_\star\brk{\dist^2(df_h,\SO{n})\circ\wexp\,\, \wexp^*\Volume}}  \,  \varpi_\star\BRK{(K_h\circ\Phi)   \,
\wpi_\star\wexp^*\Volume} \\
&\quad \stackrel{}{=} \frac{C}{h^{n}}  \int_S   (\pi_\star \Volume)\, \varpi_\star \BRK{J_{2h} \wpi_\star\brk{\dist^2(df_h,\SO{n})\circ\wexp\,\, \wexp^*\Volume} } \\
&\quad \stackrel{}{=} C  \int_S   (\pi_\star \Volume)\, \varpi_\star \BRK{K_{4h} \wpi_\star\brk{\dist^2(df_h,\SO{n})\circ\wexp\,\, \wexp^*\Volume} },
\end{split}
\]
and the latter is identical to $I_{1a}$, up to the fact that $K_{2h}$ has been replaced by $K_{4h}$.

Putting together the estimates for $I_{1a}$, $I_{1b}$, and $I_2$ we obtain
\[
\frac{I}{|\W_h|} \le Ch^2 \brk{\Eh[f_h] + 1},
\]
as required.

It remains to estimate the derivative of $\qo_h$. Writing
\[
\qo_h = \varpi_\star\BRK{K_h \,\wpi_\star\brk{\wexp^*(df_h\circ\Pi)\,\wexp^*\Volume} \circ\Psi},
\]
we differentiate,
\[
\begin{split}
\nabla\qo_h &= \varpi_\star\BRK{dK_h \, \wpi_\star\brk{\wexp^*(df_h\circ\Pi)\,\wexp^*\Volume} \circ\Psi} \\
&+ \varpi_\star\BRK{K_h \,\nabla \wpi_\star\brk{\wexp^*(df_h\circ\Pi)\,\wexp^*\Volume} \circ\Psi} \\
&+ \varpi_\star\BRK{K_h \, \wpi_\star\brk{\wexp^*(df_h\circ\Pi)\,\wexp^*\Volume} \circ\nabla\Psi}.
\end{split}
\]
The second term vanishes because $\wpi_\star\brk{\wexp^*(df_h\circ\Pi)\,\wexp^*\Volume}$ is a top-degree form. The third term
is $O(h)$ because $\nabla\Psi=0$ at the zero section of $T\S$.

Finally, noting that
\[
0 = d(1) \po_h = \varpi_\star\BRK{dK_n \,\wpi_\star\brk{(\wpi^*\varpi^*\po_h)\wexp^*\Volume}},
\]
we get
\[
\nabla\qo_h = \varpi_\star\BRK{dK_h \, \wpi_\star\brk{(\wexp^*df_h\circ\wexp^*\Pi\circ\wpi^*\Psi - \wpi^*\varpi^*\po_h)\,\wexp^*\Volume}} + O(h).
\]
Using once more Jensen's inequality, along with the bound \eqref{eq:dKh} for $|dK_h|$,
\[
|\nabla\qo_h|^2 \le \frac{C}{h^{n+2}} \varpi_\star\BRK{J_h \wpi_\star\brk{|\wexp^*df_h\circ\wexp^*\Pi\circ\wpi^*\Psi - \wpi^*\varpi^*\po_h|^2\,\wexp^*\Volume}} + O(h^2).
\]
Applying the local rigidity theorem and integrating over $\S$ we recover, up to a $1/h^2$ prefactor, the same bound as above.

\end{proof}

\medskip
Let  $f_h:\W_h\to\R^n$ satisfy the finite bending assumption \eqref{eq:fba}, and let $\qo_h:T\calM|_\S\to\R^n$ be a corresponding family of sections that by Theorem~\ref{th:rig2} satisfies
\[
\dashintW |df_h\circ\Pi - \pi^*\qo_h|^2\,\Volume \le Ch^2\BRK{\Eh[f_h] + 1}  = O(h^2),
\]
and
\[
\int_{\S} |\nabla \qo_h|^2\,\frac{\piPush\Volume}{|\W_h|} \le C \BRK{\Eh[f_h] + 1}  = O(1).
\]

By property \eqref{eq:push1b} of the push-forward operator $\piPush$ and the Cauchy-Schwarz inequality,
\[
\begin{split}
\frac{1}{|\W_h|} \int_{\S} \dist^2(\qo_h,\SO{n}) \,\piPush\Volume &=
\dashintW \dist^2(\pi^*\qo_h,\SO{n})\,\Volume \\
&\le 2 h^2\,\Eh[f_h]  + 2 \dashintW |df_h\circ\Pi - \pi^*\qo_h|^2\,\Volume \\
&= O(h^2),
\end{split}
\]
from which we deduce that
\beq
\int_{\S} |\qo_h|^2\,\frac{\piPush\Volume}{|\W_h|}  = O(1).
\label{eq:rig4}
\eeq

\section{Compactness}
\label{sec:compactness}

The results of the previous section can be summarized as follows: let $f_h\in W^{1,2}(\W_h;\R^n)$ be a family of mappings satisfying the finite bending assumption \eqref{eq:fba}. Then there exists a family of sections
$\qo_h\in W^{1,2}(T^*\calM|_\S\otimes\R^n)$, such that
\beq
\dashintW |df_h\circ\Pi - \pi^*\qo_h|^2\,\Volume = O(h^2),
\label{eq:rig6}
\eeq
\beq
\int_{\S} |\qo_h|^2\,\frac{\piPush\Volume}{|\W_h|}  = O(1),
\label{eq:rig7}
\eeq
and
\beq
\int_{\S} |\nabla\qo_h|^2\,\frac{\piPush\Volume}{|\W_h|} = O(1).
\label{eq:rig8}
\eeq

\begin{proposition}\label{pr:wcqh}
There exists a sequence (not relabeled) $\qo_h$ that weakly converges, as $h\to0$, to a limit $\qo\in W^{1,2}(\S;T^*\calM|_\S\otimes\R^n)$,  namely,
\[
\qo_h \weakly \qo \qquad\text{  in $W^{1,2}(\S;T^*\calM|_\S\otimes\R^n)$}.
\]
In particular, $\qo_h \to \qo$ strongly in $L^2(\S;T^*\calM|_\S\otimes\R^n)$.
\end{proposition}

\begin{proof}
It follows from Eqs. \eqref{eq:rig7}, \eqref{eq:rig8},
and Lemma~\ref{lm:cm},
that both $\qo_h$ and its covariant derivative are
bounded in $L^2(\S)$. Weak convergence follows from the weak-compactness of $W^{1,2}(\S;T^*\calM|_\S\otimes\R^n)$. The weak convergence of $\qo_h$ to $\qo$ in $W^{1,2}(\S;T^*\calM|_\S\otimes\R^n)$, the fact that weak boundedness implies strong boundedness, and the Sobolev embedding theorem imply that $\qo_h$ strongly converges to $\qo$ in $L^{2}(\S;T^*\calM|_\S\otimes\R^n)$.
\end{proof}

A notational convention: we will henceforth write
\[
\qperp_h = \qo_h\circ\iperp
\Textand
\qperp = \qo\circ\iperp.
\]
They belong to $W^{1,2}(\S;\NS^*\otimes\R^n)$. We write
\[
\qpar_h = \qo_h\circ\ipar
\Textand
\qpar = \qo\circ\ipar.
\]
They belong to $W^{1,2}(\S;T^*\S\otimes\R^n)$.

\begin{corollary}
\label{cor:4.1}
\beq
\begin{split}
\limH\,\, \dashintW |df_h\circ\Pi - \pi^*\qo|^2\,\Volume = 0.
\end{split}
\label{eq;df_h_to_q}
\eeq
\end{corollary}

\begin{proof}
Proposition~\ref{pr:wcqh} together with Eq. \eqref{eq:rig6} gives the desired result.
\end{proof}

\begin{proposition}
\label{prop:4.2}
$\qo\in\SO{n}$ a.e.
\end{proposition}

\begin{proof}
By the Cauchy-Schwarz inequality and the invariance of $\SO{n}$ under parallel transport,
\[
\begin{split}
\dashintW \dist^2(\pi^*\qo,\SO{n})\,\Volume &\le 2h^2\,\Eh[f_h] +
 2\dashintW |df_h\circ\Pi-\pi^*\qo|^2\,\Volume.
\end{split}
\]
Both terms on the right hand side tend to zero as $h\to0$, hence
\[
\limH\,\,\dashintW \dist^2(\pi^*\qo,\SO{n})\,\Volume  = 0.
\]
By the properties \eqref{eq:push1b},\eqref{eq:push2b} of the push-forward operator $\piPush$, and the uniform limit \eqref{eq:lim_measure},
\[
\begin{split}
0 &= \limH\,\, \dashintW \dist^2(\pi^*\qo,\SO{n})\,\Volume \\
&= \limH\,\, \frac{1}{|\W_h|} \int_\S \piPush\brk{ \dist^2(\qo,\SO{n})\circ \pi\,\Volume} \\
&= \int_{\S} \dist^2(\qo,\SO{n})\,\brk{\limH \frac{\piPush\Volume}{|\W_h|}} \\
&= \dashint_{\S} \dist^2(\qo,\SO{n})\,\VolumeS,
\end{split}
\]
which implies that $\qo\in\SO{n}$ a.e.
\end{proof}

\begin{corollary}\label{cy:df_bdd}
\beq
\dashintW |df_h|^2\,\Volume = O(1).
\label{eq:df_bdd}
\eeq
\end{corollary}

\begin{proof}
This is an immediate consequence of the finite bending assumption~\eqref{eq:fba}.
\end{proof}

Define now the averaging operator $\upi:L^1(\W_h)\to L^1(\S)$,
\[
\upi(\phi) = \frac{\piPush(\phi\,\Volume)}{\piPush\Volume},
\]
and consider the family of mappings $F_h:\S\to\R^n$ defined by
\beq
F_h = \upi(f_h).
\label{eq:Fh}
\eeq
It follows from \eqref{eq:mean0} and Lemma~\ref{lm:cm} that
\beq\label{eq:Fmean0}
\lim_{h \to 0} \,\,\dashint_\S F_h \,\VolumeS = 0.
\eeq

In the rest of this section we prove that $F_h$ strongly converges in $W^{1,2}(\S;\R^n)$ to a limit $F\in W^{2,2}(\S;\R^n)$, which is the reduced-limit of the sequence $f_h$.

\begin{lemma}
\label{lm:fpi}
\[
\dashintW |f_h - F_h \circ \pi|^2 \Volume \leq C h^2 \dashintW |df_h|^2 \Volume.
\]
\end{lemma}

\begin{proof}
By the definition of $F_h$,
\[
\upi(f_h - F_h\circ\pi) = 0.
\]
So, we apply the fibered Poincar\'e inequality,
\[
\upi(|f_h - F_h\circ\pi|^2) \le C h^2\,\upi(|df_h|^2),
\]
and integrate over $\S$ with respect to the measure $\piPush\Volume$.
\end{proof}

\begin{proposition}
\label{prop:dFh-qpar}
\beq
\dashint_\S |dF_h - \qpar_h|^2\,\VolumeS = O(h^2).
\label{eq:tmp1}
\eeq
\end{proposition}

\begin{proof}
Using Corollary~\ref{cy:h2}, Lemma~\ref{lm:etvo} and the fiberwise integrability of $f_h$,
\[
F_h = \upi(f_h) = \frac{\pi_\star( f_h \pi^\star \VolumeS \wedge \omega)}{\pi_\star(\pi^\star\VolumeS \wedge \omega)}  = \frac{\pi_\star (f_h \omega)}{\pi_\star \omega} = \frac{1}{\nu_k h^k} \pi_\star(f_h \omega) \pmod{O(h)}.
\]
So,
\beq\label{eq:dFh}
dF_h = \frac{1}{\nu_k h^k}\Brk{ \pi_\star(df_h \wedge \omega) + \pi_\star(f_h \wedge d\omega) + \left(\pi|_{\partial\W_h}\right)_\star(f_h \omega)}\pmod{O(h)}.
\eeq
By Lemma~\ref{lm:om} we have
\beq\label{eq:pWh}
\left(\pi|_{\partial\W_h}\right)_\star(f_h \omega) = 0.
\eeq
Similarly, using also the fact that $\pi_\star \omega$ is constant,
\[
\pi_\star(d\omega) = d(\pi_\star \omega)  +\left( \pi|_{\partial\W_h}\right)_{\star}\omega = 0.
\]
So,
\[
\pi_\star(f_h \wedge d\omega) = \pi_\star\left((f_h - F_h \circ \pi) d\omega\right),
\]
and
\[
dF_h - \qpar_h = \frac{1}{\nu_k h^k}\BRK{ \pi_\star\left(\left(df_h - \pi^\star \qpar_h\right)\wedge \omega\right) + \pi_\star\left((f_h - F_h \circ \pi) d\omega\right)}\pmod{O(h)}.
\]
Write $d\omega = \alpha \wedge \omega + \beta$ where $\alpha$ is a $1$-form and $\pi_\star \beta = 0.$ Choose $C$ such that $|\alpha|^2 \leq C.$ By \eqref{eq:dpisi}, $\pi$ is a Riemannian submersion with respect to $\tilde \go.$ So, by Corollary~\ref{cy:tgg}, we have
\[
|dF_h - \qpar_h| = \frac{1+O(h)}{\nu_k h^k}\BRK{ \pi_\star\left(\left|df_h - \pi^\star \qpar_h\right|\omega\right) + \sqrt{C} \pi_\star\left(|f_h - F_h \circ \pi| \omega\right)}\pmod{O(h)}.
\]
By Jensen's inequality and the Cauchy-Schwarz inequality,
\[
|dF_h - \qpar_h|^2  \leq \frac{3 + O(h)}{\nu_k h^k}\BRK{ \pi_\star\left(\left|df_h - \pi^\star \qpar_h\right|^2 \omega\right) + C \pi_\star\left(|f_h - F_h \circ \pi|^2 \omega\right)}\pmod{O(h^2)}.
\]
By Lemma~\ref{lm:pisrho} and equations~\eqref{eq:tris}, we have
\[
\pi^\star \qpar_h \circ (\sigma \oplus \iota) = (\sigma\oplus \iota)^* \circ \rho \circ \pi^* \qpar_h = \pi^* \qpar_h.
\]
So, by Lemma~\ref{lm:siPi}
\begin{align*}
\left|df_h - \pi^\star \qpar_h\right|^2 &= \left| df_h \circ \Pi - \pi^\star\qpar_h \circ \Pi \right|^2 \\
&\leq 2 \left|df_h \circ \Pi - \pi^\star \q_h \circ (\sigma \oplus \iota )\right|^2 + O(h^2)\\
&= 2 |df_h \circ \Pi - \pi^* \q_h|^2 + O(h^2).
\end{align*}
So, integrating and using Lemma~\ref{lm:cm} and Corollary~\ref{cy:h2} we obtain
\[
\begin{split}
\int_\S|dF_h - \qpar_h|^2 \frac{\VolumeS}{|\S|}
&\leq \frac{3 + O(h)}{\nu_k h^k |\S|}\int_{\W_h}\left(2\left|df_h \circ \Pi - \pi^* \q_h\right|^2 + C |f_h - F_h \circ \pi|^2\right)\eta\wedge \omega \\
&\hspace{-2cm}= (3 + O(h))\dashint_{\W_h}\left(2\left|df_h \circ \Pi - \pi^* \q_h\right|^2 + C |f_h - F_h \circ \pi|^2\right) \Volume \quad \pmod{O(h^2)}.
\end{split}
\]
The first term on the right hand side is $O(h^2)$ by \eqref{eq:rig6}. The second term is $O(h^2)$ by Lemma~\ref{lm:fpi} and~\eqref{eq:df_bdd}.

\end{proof}

\begin{corollary}\label{cy:boring}
\[
\limH\,\, \dashint_\S |dF_h - \qpar|^2\,\VolumeS = 0.
\]
\end{corollary}
\begin{proof}
The corollary follows from Proposition~\ref{prop:dFh-qpar} and the strong convergence of $\qo_h$ to $\qo$, Proposition~\ref{pr:wcqh}.
\end{proof}

\begin{lemma}\label{lm:FL2}
$F_h$ is uniformly bounded in $W^{1,2}(S,\R^n)$.
\end{lemma}

\begin{proof}
By Corollary~\ref{cy:boring}, $dF_h$ is uniformly bounded in $L^2(\S;\R^n)$. So, by the Poincar\'e inequality and Eq.~\eqref{eq:Fmean0}, $F_h$ is also uniformly bounded in $L^2(\S;\R^n)$.
\end{proof}

\begin{theorem}\label{tm:FW22}
The sequence $F_h$ has a subsequence that converges to $F\in W^{2,2}(\S;\R^n)$,
\[
F_h\to F\qquad\text{in $W^{1,2}(\S;\R^n)$}.
\]
Moreover, $dF = \qpar$.
\end{theorem}

\begin{proof}
By Lemma~\ref{lm:FL2} and the Sobolev embedding theorem, after passing to a subsequence, $F_h$  converges in $L^2(\S;\R^n)$ to a limit $F$.
By Corollary~\ref{cy:boring}, $dF_h$ converges strongly in $L^2(\S;\R^n)$ to $\qpar$. So, we conclude that $F_h$ converges strongly in $W^{1,2}(\S;\R)$. Since the limit must still be $F$, we conclude that $dF = \qpar$. Finally, since  $\qpar\in W^{1,2}(\S;T^*\S\otimes\R^n)$, it follows that $F \in W^{2,2}(\S;\R)$.
\end{proof}

\begin{corollary}
\label{cor:compact}
\[
\begin{gathered}
\limH\,\, \dashintW |f_h - F\circ\pi|^2\,\Volume = 0 \\
\limH\,\, \dashintW |df_h\circ\Pi - \pi^*(dF\oplus\qperp)|^2\,\Volume = 0.
\end{gathered}
\]
\end{corollary}
\begin{proof}
The result follows from Lemma~\ref{lm:fpi}, Corollary~\ref{cy:df_bdd}, Corollary~\ref{cor:4.1} and Theorem~\ref{tm:FW22}.
\end{proof}

To conclude, we have shown that $f_h$  has a subsequence that \rr-converges to a limit $(F,\qperp)$.

\section{Recovery sequence}
\label{sec:recovery}

Let
\[
\scrX = W^{2,2}(\S;\R^n)\times W^{1,2}(\S;\NS^*\otimes\R^n).
\]
In this section we show that for every pair $(F,\qperp)\in\scrX$ there exists a so-called recovery sequence $f_h\in W^{1,2}(\W_h;\R^n)$ that \rr-converges to $(F,\qperp)$, such that
\[
\limH \Eh[f_h] = \Elim[F,\qperp],
\]
where $\Elim$ is given by \eqref{eq:Elim}.

Let $(F,\qperp)\in \scrX$. We construct a recovery sequence $f_h\in W^{1,2}(\W_h;\R^n)$ as follows,
\beq
f_h = F\circ\pi + \pi^*\qperp \circ \lambda,
\label{eq:rec_seq}
\eeq
or in explicit form, for $\xi \in \NS$,
\[
f_h(\xi) = F(\pi(\xi)) + \qperp_{\pi(\xi)}(\xi).
\]
\begin{proposition}
\label{prop:drecovery}
The derivative of the recovery sequence satisfies
\beq
\begin{aligned}
df_h\circ\sigma &= \pi^* dF + \pi^*(\nabla\qperp)\circ \lambda  \\
df_h\circ\iota &= \pi^* \qperp,
\label{eq:df_h}
\end{aligned}
\eeq
or in explicit form, for $\xi\in\NS$ and $X = X^\parallel \oplus X^\perp \in (\pi^* T\S \oplus \pi^* \NS)_\xi \simeq (T\S \oplus \NS)_{\pi(\xi)}$,
\[
[df_h\circ(\sigma\oplus\iota)]_\xi(X^\parallel \oplus X^\perp) =
dF_{\pi(\xi)}(X^\parallel) + \qperp_{\pi(\xi)}(X^\perp) +
(\nabla_{X^\parallel}\qperp)_{\pi(\xi)}(\xi).
\]

\end{proposition}

\begin{proof}
Let $\xi\in\NS$ and $\eta\in(\NS)_{\pi(\xi)}$. Recall that
\[
\iota_\xi(\eta) = \left.\frac{d}{dt}(\xi + t\,\eta)\right|_{t=0},
\]
hence,
\[
\begin{split}
(df_h\circ\iota)_\xi(\eta) &= \left.\frac{d}{dt} f_h(\xi + t\,\eta) \right|_{t=0} \\
&= \left.\frac{d}{dt} \Brk{F(\pi(\xi)) + \qperp_{\pi(\xi)}(\xi + t\,\eta)} \right|_{t=0} \\
&=\qperp_{\pi(\xi)}(\eta),
\end{split}
\]
or in compact notation,
\[
df_h\circ\iota = \pi^* \qperp.
\]

Let then $\alpha:I\to\S$, such that $X = \dot{\alpha}(0) \in(T\S)_{\pi(\xi)}$. Recall that $\sigma_\xi(X) = \dot{\gamma}(0)$, where $\gamma:I\to\NS$  is any parallel normal field along $\alpha$. Then
\[
\begin{split}
(df_h\circ\sigma)_\xi(X)
&= \left.\frac{d}{dt} \Brk{F(\alpha(t)) + \qperp_{\alpha(t)}(\gamma(t))} \right|_{t=0} \\
&= dF_{\pi(\xi)}(X) + (\nabla_X\qperp)_{\pi(\xi)}(\xi)  + \qperp_{\pi(\xi)}\left(\left.\frac{D\gamma}{dt} \right|_{t=0}\right).
\end{split}
\]
The last term on the right hand side vanishes because $\gamma$ is parallel. So, in compact notation,
\[
df_h\circ\sigma = \pi^* dF + \pi^*(\nabla\qperp)\circ \lambda.
\]
\end{proof}

\begin{proposition}
$f_h$ \rr-converges to $(F,\qperp)$, namely,
\[
\begin{gathered}
\limH\,\, \dashintW |f_h  - F\circ\pi|^2\,\Volume = 0 \\
\limH\,\, \dashintW |df_h\circ\Pi - \pi^*(dF\oplus\qperp)|^2\,\Volume = 0.
\end{gathered}
\]
\end{proposition}

\begin{proof}
For every  $\xi\in\W_h$,
\[
(f_h - F\circ\pi)(\xi) = \qperp_{\pi(\xi)}(\xi) = O(h),
\]
hence
\[
\dashintW |f_h - F\circ\pi|^2\,\Volume = O(h^2).
\]

Next, by the Cauchy-Schwarz inequality,
\[
\begin{split}
|df_h\circ\Pi - \pi^*(dF\oplus\qperp)|^2 &\le 2\,|df_h\circ(\sigma\oplus\iota)  - \pi^*(dF\oplus\qperp)|^2  +
2\,|df_h\circ(\Pi - \sigma\oplus\iota) |^2  \\
&\le 2|\pi^*(\nabla\qperp)\circ \lambda|^2 + 2|df_h|^2\, |\Pi - \sigma\oplus\iota|^2.
\end{split}
\]
Since by Lemma~\ref{lm:siPi}, $|\Pi - \sigma\oplus\iota| = O(h)$ and $df_h$ is uniformly bounded, i.e., satisfies \eqref{eq:df_bdd}, it follows that
\[
\dashintW |df_h|^2\, |\Pi - \sigma\oplus\iota|^2\,\Volume = O(h^2).
\]
Finally, since uniformly for every $\xi\in\W_h$,
\[
|\pi^*(\nabla\qperp)\circ \lambda|(\xi) = |(\nabla\qperp)_{\pi(\xi)}(\xi)| = O(h),
\]
it follows that
\[
\dashintW |\pi^*(\nabla\qperp)\circ \lambda|^2\,\Volume = O(h^2).
\]
\end{proof}

\begin{lemma}
\label{lm:linearize_dist}
For all $A\in\GL{n}$,
\[
\Abs{\dist(\id  + A,\SO{n}) - \Abs{\frac{A + A ^T}{2}}} \le C\,\min(|A|,|A|^2).
\]
\end{lemma}

\begin{proof}
The $O(|A|)$ bound follows from the fact that all the terms on the left hand side are $O(|A|)$. The $O(|A|^2)$ bound follows form the fact that $|A+A^T|/2$ is the first-order Taylor expansion of $\dist(I+A,\SO{n})$ at $A=0$.
\end{proof}

\begin{proposition}\label{pr:limEh}
\[
\limH \,\,\Eh[f_h] = \Elim[F,\qperp].
\]
\end{proposition}

\begin{proof}
Note first that by Lemma~\ref{lm:siPi} and Proposition~\ref{prop:drecovery},
\[
\begin{split}
df_h\circ\Pi &= df_h\circ(\sigma\oplus\iota) - df_h\circ\sigma\circ\pi^*\SFF\circ(\lambda\otimes\pi^*\Ppar) + |df_h|\, O(h^2) \\
&= \pi^*(dF\oplus\qperp) + \pi^*(\nabla\qperp - dF \circ \SFF)\circ(\lambda\otimes\pi^*\Ppar) \\
& + (|\nabla\qperp| + |df_h|)\, O(h^2) \\
\end{split}
\]
By the invariance of $\SO{n}$ under parallel transport,
\[
\begin{split}
\dist(df_h,\SO{n}) &=
\dist(\pi^*(dF\oplus\qperp) + \pi^*(\nabla\qperp - dF \circ \SFF)\circ(\lambda\otimes\pi^*\Ppar) ,\SO{n}) \\
&+ (|\nabla\qperp| + |df_h|)\,O(h^2).
\end{split}
\]
Hence, by the uniform $L^2$-boundedness of $df_h$,
\[
\limH \,\, \Eh[f_h] = \limH\,\, \frac{1}{h^2}\dashintW \dist^2(\pi^*(dF\oplus\qperp) + \pi^*(\nabla\qperp-dF \circ \SFF)\circ (\lambda\otimes\pi^*\Ppar),\SO{n}) \,\Volume.
\]

Note  that for $\xi\in\W_h$, $\pi^*(\nabla\qperp - dF \circ \SFF)\circ(\lambda\otimes\pi^*\Ppar)|_\xi = O(h)$, hence
\[
\begin{split}
\limH \,\, h^2 \Eh[f_h] &= \limH\,\, \dashintW \dist(\pi^*(dF\oplus\qperp),\SO{n}) \,\Volume \\
&=  \dashint_\S \dist(dF\oplus\qperp,\SO{n}) \,\VolumeS,
\end{split}
\]
and therefore if $dF\oplus\qperp\not\in\SO{n}$ a.e., then
\[
\limH \,\, \Eh[f_h] = \infty = \Elim[F,\qperp].
\]

It remains to consider the case $\qo = dF\oplus\qperp \in\SO{n}$. We have
\[
\begin{split}
& \dist(\pi^*\qo + \pi^*(\nabla\qperp - dF \circ \SFF)\circ(\lambda\otimes\Ppar),\SO{n}) \\\
&\,\,=
\dist(\id + \pi^*(\nabla\qperp - dF \circ \SFF)\circ(\lambda\otimes\Ppar)\circ\pi^*\qo^{-1},\SO{n}).
\end{split}
\]
It follows from
Lemma~\ref{lm:linearize_dist} that for $A\in\GL{n}$ with $|A| = O(h)$,
\[
\dist^2(\id + A,\SO{n}) = \Abs{\frac{A+A^T}{2}}^2 + O(h^3).
\]
Take
\beq
A = \pi^*(\nabla\qperp - dF \circ \SFF)\circ (\lambda\otimes\pi^*\Ppar) \circ\pi^*\qo^{-1},
\label{eq:A}
\eeq
or more explicitly, for $v \in \R^n$
\[
A_\xi v = (\nabla_{(\Pparq)_{\pi(\xi)}(v)}\qperp)_{\pi(\xi)}(\xi) - (dF \circ \SFF)_{\pi(\xi)}(\xi,(\Pparq)_{\pi(\xi)}(v)).
\]
Using the fact that $A_\xi = O(h)$, we obtain
\begin{align}
\limH \,\,\Eh[f_h] &= \limH\,\,\frac{1}{h^2} \dashintW \Abs{\frac{A+A^T}{2}}^2 \,\Volume \notag\\
&= \limH\,\,\frac{1}{2 h^2} \dashintW \tr(A^T A + A^2)\,\Volume. \label{eq:EhA}
\end{align}

At this point, it is helpful to introduce index notation to clarify the sense in which tensors with several upper and lower indices are composed and transposed. Let $p \in\S$, and let $X_1,\ldots,X_n$, be a basis for $T_p\calM$ such that $X_1,\ldots,X_m$ is a basis for $T_p\S$, and $X_{m+1},\ldots,X_{m+n}$, is a basis for $\NS_p$. To keep our notation concise, we use the following convention for ranges of summation:
\begin{align*}
&i,j,k, \qquad && \text{run from $1$ to $n$,} \\
&a,b,c, \qquad && \text{run from $1$ to $m$,} \\
&u,v,w, \qquad && \text{run from $m+1$ to $n$.}
\end{align*}
The orthogonality of $T_p\S$ and $\NS_p$ implies that $\go_{au}=0$.

Let $Y_1,\ldots,Y_n$, be a basis of $\R^n$. We reserve Greek letters for indices associated to the $Y's$. For $\xi \in \NS_p$ and $\eta \in \R^n$, write
\[
\xi = \xi^u X_u, \qquad \eta = \eta^\alpha Y_\alpha.
\]
So, for example,
\[
A_\xi \eta = \brk{\nabla_a \q_u^\alpha -  \qo^\alpha_c \SFF^c_{au}} (\q^{-1})^a_\beta \xi^u  \eta^\beta Y_\alpha.
\]

By equation~\eqref{eq:EhA} and Lemma~\ref{lm:cm} we get that
\[
\limH \Eh[f_h] =  \dashint_\S \calW \;\VolumeS,
\]
where
\[
\calW = \limH \left(\frac{1}{2 h^2} \upi(\tr(A^T A + A^2))\right).
\]
We now calculate $\calW$ explicitly.
Define the section $M$ of $\NS \otimes \NS$ by
\[
M = \limH \,\,\frac{1}{h^2} \upi (\lambda \otimes \lambda).
\]
At the point $p\in\S$, we have
\[
\begin{split}
\calW &= \half \euc_{\alpha\gamma} \go^{ab}  (\nabla_a \qo_u^\alpha -  \qo^\alpha_c \SFF^c_{au}) (\nabla_b \qo_t^\gamma -  \qo^\gamma_d \SFF^d_{bt}) M^{ut}  \\
&+ \half  (\nabla_a \qo_u^\alpha -  \qo^\alpha_c \SFF^c_{au})  (\qo^{-1})_\beta^a  (\nabla_b \qo_t^\beta -  \qo^\beta_d \SFF^d_{bt})  (\qo^{-1})_\alpha^b M^{ut}.
\end{split}
\]
Since we have chosen $\W_h$ to be the set of all points of distance less than $h$ from $\S$, we have $M^{uv} = \kappa\go^{uv}$, where $\kappa$ is the volume of the $k-1$ dimensional unit sphere divided by $k+2$.
Hence the reduced energy density is
\beq
\begin{split}
\calW &= \frac{\kappa}{2} \euc_{\alpha\gamma} \go^{ab} \go^{ut}
(\nabla_a \qo_u^\alpha -  \qo^\alpha_c \SFF^c_{au})
(\nabla_b \qo_t^\gamma -  \qo^\gamma_d \SFF^d_{bt})  \\
&+
\frac{\kappa}{2}
\go^{ut}  ((\qo^{-1})_\alpha^b\nabla_a \qo_u^\alpha -  \SFF^b_{au})
((\qo^{-1})_\beta^a \nabla_b \qo_t^\beta - \SFF^a_{bt}).
\end{split}
\label{eq:calW}
\eeq

Since  $(dF)_a^\alpha = \qo_a^\alpha$, it follows that
\beq\label{eq:qsym}
\nabla_c \qo_a^\alpha = \nabla_a \qo_c^\alpha.
\eeq
Differentiating the equation $0 = \Ppar \circ \iperp = \Ppar \circ \qo^{-1} \circ \qo \circ \iperp
= \Pparq\circ\qperp$, we obtain
\[
0 = \nabla\Pparq \circ \qperp + \Pparq \circ \nabla \qperp,
\]
or, in index notation,
\beq\label{eq:dbp}
(\nabla_a \qo_u^\alpha) (\qo^{-1})_\alpha^b = -(\nabla_a (\qo^{-1})_\alpha^b) \qo_u^\alpha.
\eeq
By equations~\eqref{eq:qsym}, and~\eqref{eq:dbp}, the fact the metrics are covariant constants and the orthogonality of $\qo$,
\[
\begin{aligned}
(\nabla_a \qo_u^\alpha) (\qo^{-1})_\alpha^b &= -(\nabla_a (\qo^{-1})_\alpha^b) \qo_u^\alpha =
-\euc_{\alpha\gamma}\go^{bc}  (\nabla_a \qo^\gamma_c) \qo_u^\alpha \\
&= -\euc_{\alpha\gamma}\go^{bc}  (\nabla_c \qo^\gamma_a) \qo_u^\alpha =
\euc_{\alpha\gamma}\go^{bc}  (\nabla_c \qo_u^\alpha) \qo^\gamma_a.
\end{aligned}
\]
Substituting this last identity as well as the symmetry of the second fundamental form into the second term of \eqref{eq:calW},
\[
\begin{split}
\go^{ut}  ((\qo^{-1})_\alpha^b\nabla_a \qo_u^\alpha -  \SFF^b_{au})
((\qo^{-1})_\beta^a \nabla_b \qo_t^\beta - \SFF^a_{bt}) &= \\
&\hspace{-4cm} =
\go^{ut}  (\euc_{\alpha\gamma}\go^{bc}  (\nabla_c \qo_u^\alpha) \qo^\gamma_a -  \go^{bc} \go_{ad} \SFF^d_{cu})
((\qo^{-1})_\beta^a \nabla_b \qo_t^\beta - \SFF^a_{bt}) \\
&\hspace{-4cm} =
\go^{ut}  \euc_{\alpha\gamma}\go^{bc}   ((\nabla_c \qo_u^\alpha) -
\qo^\alpha_d   \SFF^d_{cu})
(\qo^\gamma_a(\qo^{-1})_\beta^a \nabla_b \qo_t^\beta - \qo^\gamma_a\SFF^a_{bt}) \\
&\hspace{-4cm} = |\Pparq\circ(\nabla\qperp - \qpar\circ\SFF)|^2.
\end{split}
\]
Thus,
\[
\calW = \frac{\kappa}{2} |\nabla \q^\perp - \qpar\circ\SFF|^2 +
\frac{\kappa}{2}  |\Pparq \circ (\nabla \qperp  - \qpar\circ\SFF)|^2.
\]
Noting that $|\cdot|^2 = |\Pparq\circ\cdot| +  |\Pperpq\circ\cdot|$, $\Pparq\circ\qpar=\id$, and $\Pperpq\circ\qperp=0$, we have
\[
\calW = \kappa  |\Pparq \circ \nabla \qperp  - \SFF|^2 +
\frac{\kappa}{2} |\Pperpq\circ \nabla \q^\perp|^2.
\]
The proposition follows immediately.
\end{proof}

\section{Lower semicontinuity}

In this section we show that a sequence $f_h$ that satisfies the finite bending assumption \eqref{eq:fba} and whose \rr-limit is $(F,\qperp)$, satisfies the lower-semicontinuity property,
\[
\liminfH \Eh[f_h] \ge \Elim[F,\qperp].
\]
We first pass to a subsequence $f_{h_k}$ so that
\[
\lim_{k \to \infty} \calE_{h_k} [f_{h_k}] = \liminfH \Eh[f_h].
\]
Thus, in the following arguments, we may freely pass to a further subsequence; from now on, we drop the subscript $k$.

Let $\qo_h$ be a sequence obtained from the rigidity Theorem~\ref{th:rig2}, and let $\tqo_h$ be an orthogonal projection of $\qo_h$ on $\SO{n}$, i.e., $\tqo_h\in\SO{n}$ and $|\qo_h-\tqo_h|=\dist(\qo_h,\SO{n})$. Applying the Cauchy-Schwarz inequality and the invariance of $\SO{n}$ under parallel transport,
\[
|\pi^*\qo_h - \pi^*\tqo_h|^2 \le 2\,|df_h\circ\Pi - \pi^*\qo_h|^2 + 2\,\dist^2(df_h,\SO{n}).
\]
Averaging over $\W_h$,  using once more the properties of the push-forward operator $\piPush$, estimate~\eqref{eq:rig6}, the  finite bending assumption~\eqref{eq:fba}, and the estimate of the volume form discrepancy, Lemma~\ref{lm:cm},
we have
\beq
\dashint_\S |\qo_h - \tqo_h|^2\,\VolumeS = O(h^2).
\label{eq:qo-tqo}
\eeq
It follows from Eqs.~\eqref{eq:rig6} and \eqref{eq:qo-tqo} that
\beq
\dashintW |df_h\circ\Pi - \pi^*\tqo_h|^2\,\Volume = O(h^2).
\label{eq:tqo}
\eeq
Passing to a subsequence and using the $L^2$-convergence of $\qo_h$ to $\qo$ (Proposition~\ref{pr:wcqh}), we have
\[
\limH \,\,\dashint_\S |\tqo_h - \qo|^2\,\VolumeS =  0.
\]
Note  that we can only guarantee the convergence of $\tqo_h$ to $\qo$ in $L^2(\S;T^*\calM|_\S\otimes\R^n)$, whereas $\qo_h$ converges to $\qo$ weakly in $W^{1,2}(\S;T^*\calM|_\S\otimes\R^n)$. The reason for defining the possibly less regular sequence $\tqo_h$ will be made clear further below.

\begin{proposition}
The sequence
\beq
\ao_h = \frac{\qo_h - \tqo_h}{h}
\label{eq:ah}
\eeq
has a subsequence (not relabeled) that weakly converges, as $h\to0$, in $L^2(\S;T^*\calM|_\S\otimes\R^n)$; we denote the limit by $\ao$.
\end{proposition}

\begin{proof}
This is an immediate consequence of Eq.~\eqref{eq:qo-tqo}, as it follows that the sequence $\ao_h$ is bounded in $L^2(\S;T^*\calM|_\S\otimes\R^n)$.
\end{proof}

\begin{proposition}
\label{prop:dfh-dhfh}
Let $F_h = \upi(f_h)$ and define
\[
\hf_h = F_h\circ\pi + \pi^*\qperp_h\circ\lambda.
\]
Then
\[
\limH\,\,\dashintW |df_h - d\hf_h|^2\,\Volume = O(h^2).
\]
\end{proposition}

\begin{proof}
Compare the definition of $\hf_h$ with that of the recovery sequence \eqref{eq:rec_seq}.
By an argument similar to that used in Proposition~\ref{prop:drecovery} we find
\beq
\begin{aligned}
d\hf_h\circ\sigma &= \pi^*dF_h + \pi^*(\nabla\qperp_h)\circ\lambda \\
d\hf_h\circ\iota &= \pi^*\qperp_h.
\end{aligned}
\label{eq:dhfh}
\eeq
Using the invariance of the inner-product under parallel transport, along with the Cauchy-Schwarz inequality (twice) and Lemma~\ref{lm:siPi},
\[
\begin{split}
|df_h - d\hf_h|^2 &= |df_h\circ\Pi - d\hf_h\circ\Pi|^2 \\
&\le 2\,|df_h\circ\Pi - d\hf_h(\sigma\oplus\iota) |^2 + 2\,|d\hf_h|^2 |(\sigma\oplus\iota)  - \Pi|^2 \\
&\le 4\,|df_h\circ\Pi - \pi^*(dF_h\oplus\qperp_h)|^2 + 4\,|\pi^*(\nabla\qperp_h)\circ\lambda|^2 + 2\,|d\hf_h|^2 \,O(h^2).
\end{split}
\]
We then average over $\W_h$.
The first term on the right hand side is $O(h^2)$ by Eq.~\eqref{eq:rig6}, Proposition~\ref{prop:dFh-qpar}, and Lemma~\ref{lm:cm}. The second term is $O(h^2)$ because  $\nabla \qo_h$ is bounded in $L^2$ by Eq.~\eqref{eq:rig8} and $|\lambda| = O(h)$. Finally, the third term is $O(h^2)$ because $|d\hf_h|^2$ is uniformly bounded in $L^2(\W_h)$ as obtained by combining Eq.~\eqref{eq:dhfh}, Lemma~\ref{lm:siPi}, Lemma~\ref{lm:FL2}, Eq.~\eqref{eq:rig7} and Eq.~\eqref{eq:rig8}.
\end{proof}

\begin{proposition}
\label{prop:bh}
Let $\hf_h$ be defined as above.
Then the sequence
\beq
\bo_h = \frac{df_h - d\hf_h}{h}
\label{eq:bh}
\eeq
weakly converges to zero in the sense of Definition~\ref{df:wc}.
\end{proposition}

\begin{proof}
Let $\Phi\in C_0^\infty(\W_{h_0};T\W_{h_0}\otimes\R^n)$. Denote by $d^{\tilde *}$ the dual of the exterior derivative with respect to the $L^2$ pairing of $\tilde \go$. By Lemma~\ref{lm:star}, $d^{\tilde *}$ is also the  dual of the exterior derivative with respect to the $L^2$ pairing of $\mu_h^\star \go$. Integrating by parts,
\[
\int_{\W_{h_0}} (\mu_h^\star\tilde \go)(\mu_h^*(df_h - d\hf_h),\Phi)\,\Volumet = \int_{\W_{h_0}} (\mu_h^\star\tilde\go)((f_h - \hf_h)\circ \mu_h ,d^{\tilde*} \Phi)\,\Volumet.
\]
Using the Cauchy-Schwarz inequality and Lemma~\ref{lm:resc}
\begin{align*}
\Abs{\int_{\W_{h_0}} (\mu_h^\star\tilde \go)(\mu_h^*(df_h - d\hf_h),\Phi)\,\Volumet}^2 &\le
C'(\Phi) \,\brk{\int_{\W_{h_0}} |f_h - \hf_h|^2\circ \mu_h\,\Volumet} \\
&\leq C(\Phi) \, \brk{\dashintW |f_h-\hf_h|^2 \Volume}.
\end{align*}
Using the fact that $\upi(\hf_h) = \upi(f_h)$ and applying the fibered Poincar\'{e} inequality,
\[
\begin{split}
\Abs{\int_{\W_{h_0}}(\mu_h^\star\tilde \go)(\mu_h^*(df_h - d\hf_h),\Phi)\,\Volume}^2 &\le
\frac{C(\Phi)}{|\W_h|} \int_\S \upi\brk{|f_h - \hf_h|^2}\,\pi_\star\Volume \\
&\le \frac{C(\Phi) h^2}{|\W_h|} \int_\S \upi\brk{|df_h - d\hf_h|^2}\,\pi_\star\Volume \\
&= C(\Phi) h^2  \, \dashintW |df_h - d\hf_h|^2\,\Volume.
\end{split}
\]
Dividing by $h^2$ and applying Proposition~\ref{prop:dfh-dhfh},
\begin{align*}
\limH\,\,\Abs{\int_{\W_{h_0}} (\mu_h^\star\tilde \go)\left(\mu_h^*\left(\frac{df_h - d\hf_h}{h}\right),\Phi\right)\,\Volumet}^2 &\le
C(\Phi) \, \limH\,\,\dashintW |df_h - d\hf_h|^2\,\Volume \\
&= 0.
\end{align*}
Since $C^\infty(\W_{h_0};T^*\W_{h_0}\otimes\R^n)$ is dense in $L^2(\W_{h_0};T^*\W_{h_0}\otimes\R^n)$, this equation holds for all $\Phi\in L^2(\W_{h_0};T^*\W_{h_0}\otimes\R^n)$.
\end{proof}

\begin{proposition}
The sequence
\beq
\co_h = \frac{dF_h - \qpar_h}{h}
\label{eq:ch}
\eeq
has a subsequence (not relabeled) that weakly converges,  as $h\to0$, in $L^2(\S;T^*\S\otimes\R^n)$; we denote the limit by $\co$.
\end{proposition}

\begin{proof}
This is an immediate consequence of Proposition~\ref{prop:dFh-qpar}, which together with Lemma~\ref{lm:cm}
implies that
\[
\dashint_\S |dF_h - \qpar_h|^2\,\VolumeS = O(h^2),
\]
hence $\co_h$ is a bounded sequence in $L^2(\S;T^*\S\otimes\R^n)$.
\end{proof}

\medskip
We now turn to estimate $\Eh[f_h]$ as $h\to0$.
By the invariance of $\SO{n}$ under parallel transport and the fact that $\tqo_h\in\SO{n}$,
\[
\begin{split}
\dist^2(df_h,\SO{n}) &= \dist^2(df_h\circ\Pi,\SO{n}) \\
&=  \dist^2(df_h\circ\Pi\circ\pi^*\tqo_h^{-1},\SO{n}) \\
&= \dist^2\brk{I + h\,\Gh,\SO{n}},
\end{split}
\]
where $\Gh\in L^2(\W_h;\GL{n})$ is given by
\beq
\Gh = \frac{df_h\circ\Pi\circ\pi^*\tqo_h^{-1} - \id}{h}.
\label{eq:defGh}
\eeq
Thus,
\[
\Eh[f_h] = \frac{1}{h^2} \dashintW \dist^2(\id + h\,\Gh,\SO{n})\,\Volume.
\]

We start by making a few observations about $\Gh$:

\begin{proposition}
\label{prop:Gh_bdd}
The sequence $G_h\in L^2(\W_h;\R^n\otimes\R^n)$ is bounded, namely,
\[
\dashintW |\Gh|^2\,\Volume = O(1).
\]
\end{proposition}

\begin{proof}
This is an immediate consequence of the estimate \eqref{eq:tqo}, as
\[
\begin{split}
\dashintW |\Gh|^2\,\Volume &= \frac{1}{h^2} \dashintW |df_h\circ\Pi\circ\pi^*\tqo_h^{-1} - \id|^2\,\Volume \\
&= \frac{1}{h^2} \dashintW |df_h\circ\Pi - \pi^*\tqo_h|^2\,\Volume =  O(1).
\end{split}
\]
\end{proof}

\begin{proposition}
\label{prop:Ghexpression}
$\Gh$ can be expressed in the following form:
\[
\begin{split}
\Gh  &=
\bo_h \circ\Pi\circ\pi^*\tqo_h^{-1} +
\pi^*((\ao_h + \co_h\circ\Ppar) \circ\tqo_h^{-1}) \\
&+\frac{d\hf_h\circ(\Pi  - \sigma\oplus\iota + \sigma\circ\pi^*\SFF\circ(\lambda\otimes\pi^*\Ppar))}{h} \circ(\pi^*\tqo_h^{-1}) \\
&+ \pi^*(\nabla\qperp_h - dF_h\circ\SFF)\circ\brk{\frac{\lambda}{h}\otimes\pi^*\Ppar}
 \circ\pi^*\tqo_h^{-1} \\
&
- \pi^*(\nabla\qperp_h)\circ\left(\lambda \otimes\pi^*\SFF\circ\brk{\frac{\lambda}{h}\otimes\pi^*\Ppar}
\right)  \circ\pi^*\tqo_h^{-1}.
\end{split}
\]
\end{proposition}

\begin{proof}
This follows after straightforward algebraic manipulations from  the definitions \eqref{eq:ah} and \eqref{eq:bh} of $\ao_h$ and $\bo_h$,  formula
\eqref{eq:dhfh} for $d\hf_h\circ(\sigma\oplus\iota)$, and  the definition \eqref{eq:ch} of $\co_h$.
\end{proof}

\begin{proposition}
Let
\[
\tEh[f_h] = \frac{1}{h^2} \dashintW \chi_h\,\dist^2(\id + h\,\Gh,\SO{n})\,\Volume,
\]
where
\[
\chi_h = \Cases{1 & |G_h| < h^{-1/4} \\ 0 & \text{otherwise},}
\]
and let
\[
J_h[f_h] = \dashint \chi_h \, \Abs{\frac{\Gh + \Gh^T}{2}}^2\,\Volume.
\]
Then,
\[
\limH \brk{J_h[f_h] - \tEh[f_h]} = 0.
\]
\end{proposition}

\begin{proof}
By definition,
\[
J_h[f_h] - \tEh[f_h] = \dashintW \chi_h \brk{ \alpha_h^2 - \beta_h^2}\,\Volume.
\]
where
\[
\alpha_h = \frac{\dist(\id + h\,\Gh,\SO{n})}{h}
\Textand
\beta_h = \Abs{\frac{\Gh + \Gh^T}{2}}.
\]
Using Lemma~\ref{lm:linearize_dist} with $A$ replaced by $h\,\Gh$,
we obtain
\beq
|\alpha_h - \beta_h| \le C \min(|\Gh|,h|\Gh|^2).
\label{eq:ineq_ab}
\eeq

When the indicator function $\chi_h$ is non-zero, $|\Gh| <  h^{-1/4}$, hence
\[
|\alpha_h^2 - \beta_h^2| \le |\alpha_h - \beta_h| \,(|\alpha_h - \beta_h| + \beta_h) \le C h |\Gh|^2 \cdot 2 |\Gh|
=  O(h^{1/4}),
\]
from which follows that
\[
|J_h[f_h] - \tEh[f_h]| \le
\dashintW \chi_h |\alpha_h^2 - \beta_h^2|\,\Volume = O(h^{1/4}),
\]
which completes the proof.
\end{proof}

\medskip
Since, trivially, $\tEh[f_h] \le \Eh[f_h]$, if follows that
\[
\liminfH \,\, \Eh[f_h] \ge \liminfH\,\, J_h[f_h],
\]
hence it only remains to show that
\[
\liminfH\,\, J_h[f_h]  \ge \Elim[F,\qperp].
\]

To this end we write $\Gh = \Gh^{(1)} + \Gh^{(2)}$, where
\[
\Gh^{(1)} = \Gh - \frac{A}{h}  - \pi^*[(\co \circ\Ppar + \ao)\circ\qo^{-1}],
\]
and
\[
\Gh^{(2)} = \frac{A}{h}  + \pi^*[(\co \circ\Ppar + \ao)\circ\qo^{-1}],
\]
where $A$ is given by Eq.~\eqref{eq:A}, namely,
\[
A = \pi^*(\nabla\qperp - \qpar\circ\SFF)\circ (\lambda\otimes\pi^*\Ppar)\circ\pi^*\qo^{-1}.
\]

Thus,
\[
\begin{aligned}
J_h[f_h] &= \frac14 \,\,\dashintW \chi_h\,|\Gh^{(1)} + {\Gh^{(1)}}^T|^2 \,\Volume \\
&+ \frac14 \,\,\dashintW \chi_h\,|\Gh^{(2)} + {\Gh^{(2)}}^T|^2\,\Volume \\
&+ \frac12 \,\,\dashintW \chi_h\,\euc\brk{\Gh^{(1)} + {\Gh^{(1)}}^T,\Gh^{(2)} + {\Gh^{(2)}}^T},\Volume,
\end{aligned}
\]
and
\beq
\begin{aligned}
\liminfH\,\,J_h[f_h] &\ge  \liminfH\,\,\frac14 \,\,\dashintW \chi_h\,|\Gh^{(2)} + {\Gh^{(2)}}^T|^2\,\Volume \\
&+ \liminfH\,\,\frac12 \,\,\dashintW \chi_h\,\euc\brk{\Gh^{(1)} + {\Gh^{(1)}}^T,\Gh^{(2)} + {\Gh^{(2)}}^T}\,\Volume,
\end{aligned}
\label{eq:liminf_ineq}
\eeq

\begin{proposition}
\label{prop:Gh2a}
\[
\liminfH\,\,\frac14 \,\,\dashintW \chi_h\,|\Gh^{(2)} + {\Gh^{(2)}}^T|^2\,\Volume =
\liminfH\,\,\frac14 \,\,\dashintW |\Gh^{(2)} + {\Gh^{(2)}}^T|^2\,\Volume.
\]
\end{proposition}

\begin{proof}
Consider the difference
\[
\Delta_h = \dashintW (1-\chi_h)\,|\Gh^{(2)} + {\Gh^{(2)}}^T|^2\,\Volume.
\]
Using the Cauchy-Schwarz inequality (twice), the uniform bound $|\lambda/h|\le 1$, and Lemma~\ref{lm:resc},
\beq\label{eq:del}
\begin{split}
\Delta_h &\le 4\,\,\dashintW (1-\chi_h) |\Gh^{(2)}|^2\,\Volume \\
&\le 8\,\, \dashintW (1-\chi_h) \pi^* \brk{|\nabla\qperp|^2 + |(\co\circ\Ppar + \ao)|^2
+ |(\qpar\circ\SFF)|^2}\,\Volume \\
&\leq C \int_{\W_{h_0}} ((1-\chi_h)\circ \mu_h)\, \pi^* \brk{|\nabla\qperp|^2 + |(\co\circ\Ppar + \ao)|^2
+ |(\qpar\circ\SFF)|^2}\, \eta\wedge\omega.
\end{split}
\eeq
Using Lemma~\ref{lm:resc} and Proposition~\ref{prop:Gh_bdd},
\begin{align*}
\int_{\W_{h_0}}((1-\chi_h)\circ\mu_h)\eta\wedge\omega  &\leq C' \dashintW (1-\chi_h) \Volume \\
&\leq C' h^{1/2}\dashintW |\Gh|^2\,\Volume = O(h^{1/2}).
\end{align*}
It follows that $(1-\chi_h)\circ\mu_h$ is the indicator function of a set of measure tending to zero with $h$.
So, since $\pi^* \brk{|\nabla\qperp|^2 + |(\co\circ\Ppar + \ao)|^2
+ |(\qpar\circ\SFF)|^2}$ is integrable in $\W_{h_0}$, inequality~\eqref{eq:del} implies that
$\Delta_h$ tends to zero as $h\to0$, which completes the proof.
\end{proof}

\begin{proposition}
\label{prop:Gh2b}
\[
\liminfH\,\,\frac14 \,\,\dashintW |\Gh^{(2)} + {\Gh^{(2)}}^T|^2\,\Volume \ge \Elim[F,\qperp].
\]
\end{proposition}

\begin{proof}
The integral
\[
\frac14 \,\,\dashintW |\Gh^{(2)} + {\Gh^{(2)}}^T|^2\,\Volume
\]
depends on the sections $\ao$ and $\co$. It is easy to see that for every finite $h>0$ this integral is minimal for $\ao,\co = O(h^2)$, hence
\[
\liminfH\,\,\frac14 \,\,\dashintW |\Gh^{(2)} + {\Gh^{(2)}}^T|^2\,\Volume \ge
\limH\,\, \frac{1}{h^2} \dashintW \Abs{\frac{A + A^T}{2}}^2\,\Volume = \Elim[F,\qperp],
\]
and the last identity was proved in Section~\ref{sec:recovery}.
\end{proof}

\begin{proposition}
\label{prop:Gh12a}
\[
\begin{split}
&\liminfH\,\,\frac12 \,\,\dashintW \chi_h\,\euc\brk{\Gh^{(1)} + {\Gh^{(1)}}^T,\Gh^{(2)} + {\Gh^{(2)}}^T}\,\Volume = \\
&\quad =  \liminfH\,\,\frac12 \,\,\dashintW \euc\brk{\Gh^{(1)} + {\Gh^{(1)}}^T,\Gh^{(2)} + {\Gh^{(2)}}^T}\,\Volume.
\end{split}
\]
\end{proposition}

\begin{proof}
Consider the difference
\[
\begin{split}
\Delta_h &= \dashintW (1-\chi_h) \euc\brk{\Gh^{(1)} + {\Gh^{(1)}}^T,\Gh^{(2)} + {\Gh^{(2)}}^T}\,\Volume.
\end{split}
\]
Using the bilinearity of $\go$ and the Cauchy-Schwarz inequality,
\[
|\Delta_h|^2 \le 4 \brk{\dashintW |\Gh^{(1)}|^2\,\Volume} \brk{\dashintW |\chi_h \Gh^{(2)}|^2\,\Volume}.
\]
The first term on the right hand side is uniformly bounded by Proposition~\ref{prop:Gh_bdd} whereas the second term tends to zero by  the same argument as in the proof of Proposition~\ref{prop:Gh2a}.
\end{proof}

\begin{proposition}
\label{prop:Gh12b}
\[
\limH\,\,\frac12 \,\,\dashintW \euc\brk{\Gh^{(1)} + {\Gh^{(1)}}^T,\Gh^{(2)} + {\Gh^{(2)}}^T}\,\Volume = 0.
\]
\end{proposition}

\begin{proof}
Note that
\[
\euc\brk{\Gh^{(1)} + {\Gh^{(1)}}^T,\Gh^{(2)} + {\Gh^{(2)}}^T} =
\go\brk{(\Gh^{(1)} + {\Gh^{(1)}}^T)\circ\pi^*\qo,(\Gh^{(2)} + {\Gh^{(2)}}^T)\circ\pi^*\qo}.
\]
Using Proposition~\ref{prop:Ghexpression}, the sections $\Gh^{(1)}\circ \pi^*\qo$ can be rewritten as follows:
\[
\begin{split}
\Gh^{(1)}\circ \pi^* \qo  &=
\bo_h \circ\Pi\circ \pi^*(\tqo_h^{-1}\circ \qo)   \\
&+ \pi^*[(\ao_h + \co_h\circ\Ppar) \circ\tqo_h^{-1}\circ \qo] - \pi^*[\ao + \co \circ\Ppar] \\
&+ \pi^*\nabla\qperp_h \circ \brk{\frac{\lambda}{h}\otimes\pi^*\Ppar}  \circ\pi^*(\tqo_h^{-1} \circ \qo)
- \pi^*\nabla\qperp\circ \brk{\frac{\lambda}{h} \otimes\pi^*\Ppar} \\
&- \pi^* (dF_h\circ\SFF)\circ\brk{\frac{\lambda}{h}\otimes\pi^*\Ppar}  \circ\pi^*(\tqo_h^{-1}\circ \qo)
+ \pi^*(\qpar\circ\SFF)\circ\brk{\frac{\lambda}{h}\otimes\pi^*\Ppar} \\
&+\frac{d\hf_h\circ(\Pi  - \sigma\oplus\iota + \sigma\circ\pi^*\SFF\circ(\lambda\otimes\pi^*\Ppar))}{h} \circ\pi^*(\tqo_h^{-1}\circ \qo) \\
& - \pi^*(\nabla\qperp_h)\circ\left(\lambda \otimes \pi^*\SFF\circ\brk{\frac{\lambda}{h}\otimes\pi^*\Ppar}
\right)  \circ\pi^* (\tqo_h^{-1}\circ \qo).
\end{split}
\]
Consider the first line of the right hand side:  $\bo_h$ weakly converges to zero in the sense of Definition~\ref{df:wc} (Proposition~\ref{prop:bh}), whereas $\tqo_h^{-1}$ is bounded in $L^\infty(\S;\R^n\otimes T\calM|_\S)$ and strongly converges to $\qo^{-1}$ in $L^2(\S;\R^n\otimes T\calM|_\S)$. It is a known fact that the product of an $L^2$ weakly converging sequence and an $L^2$ strongly convergent sequence that is bounded in $L^\infty$ weakly converges in $L^2$ to the product of the limits. We therefore conclude that the first line weakly converges to zero in the sense of Definition~\ref{df:wc}. By a similar argument the second, third, and fourth lines weakly converge to zero as well. For the fourth line, we use also Proposition~\ref{prop:dFh-qpar}. The fifth line is $O(h)$ by Lemma~\ref{lm:siPi}, whereas the sixth line is $O(h)$ due to the $\lambda$-factor. Thus,  $\Gh^{(1)} \circ \pi^*\qo$ weakly converges to zero in the sense of Definition~\ref{df:wc}. Similarly, $\Gh^{(1)^T} \circ \pi^*\qo$ weakly converges to zero.
Noting that
\[
(\Gh^{(2)} + {\Gh^{(2)}}^T)\circ\pi^*\qo = \pi^*\beta_1 + \pi^*\beta_2\circ\frac{\lambda}{h}
\]
for suitable $\beta_1,\beta_2$,
we apply Lemma~\ref{lm:wc} to obtain the desired result.
\end{proof}

\begin{corollary}
\[
\liminfH\,\, J_h[f_h]  \ge \Elim[F,\qperp].
\]
\end{corollary}

\begin{proof}
This is an immediate consequence of Eq.~\eqref{eq:liminf_ineq} and Propositions~\ref{prop:Gh2a}, \ref{prop:Gh2b},
\ref{prop:Gh12a}, and  \ref{prop:Gh12b}.
\end{proof}

\begin{corollary}
\[
\liminfH\,\, \Eh[f_h] \ge \Elim[F,\qperp].
\]
\end{corollary}

\section{Examples}
\label{sec:examples}

\subsection{Plates and shells}

Plates and shells correspond to the case of $n=3$ and $k=1$, that is, the limiting manifold $\S$ is a two-dimensional surface.  For $k=1$, the limiting energy functional $\Elim$ further simplifies, as the second term in \eqref{eq:Elim} vanishes. This is because for $k=1$,
\beq\label{eq:ppqnqp}
\Pperpq \circ\nabla \qperp=0.
\eeq
Indeed, let $u$ be a local unit length section of $\NS.$ Then, $\nabla u = 0$ because
\[
 0 = d \go(u,u) = 2 \go(u,\nabla u)
\]
and $\NS$ is one-dimensional. So,
\[
(\nabla \qperp)(u) = d(\qperp(u)) - \qperp(\nabla u) = d (\qperp(u)).
\]
On the other hand, since $\Pperp$ is an orthogonal projection, and $\qo$ is orthogonal,
\begin{align*}
&\go(\Pperpq \circ (\nabla \qperp)(u),u) = \euc((\nabla \qperp)(u),\qperp(u)) = \\
& \qquad \qquad \qquad \qquad = \euc(d(\qperp(u)),\qperp(u)) = \frac{1}{2} d\euc(\qperp(u),\qperp(u)) = 0.
\end{align*}
Since $u$ spans $\NS,$ equation~\eqref{eq:ppqnqp} follows.

Finally, we note that since $\qo\in\SO{3}$, it follows that $\qperp$ is unambiguously determined by $dF$, which means that the $F:\S\to\R^3$ fully characterizes the limiting configuration.
Noting that $\kappa=1/3$, the limiting functional for plates and shells is:
\beq
\Elim[F] = \frac{1}{3}\,\, \dashint_\S |\Pparq\circ\nabla \qperp  - \SFF|^2 \,\VolumeS.
\label{eq:plates}
\eeq

This integral has a well-defined physical meaning. It is the mean square difference between the induced second fundamental form $\Pparq\circ\nabla \qperp$ (the gradient of the normal to the surface) and the intrinsic second fundamental form, i.e., it is a bending energy.

The limiting energy \eqref{eq:plates} applies equally to  plates and shells, the only difference being whether $\SFF=0$ (plates) or $\SFF\ne0$ (shells). It also applies equally to Euclidean and non-Euclidean plates/shells. The plate/shell is Euclidean if there exists an immersion $F:\S\to\R^3$, such that $F^*\euc = \go|_\S$ and $\Pparq\circ\nabla \qperp  = \SFF$, or equivalently, if $\go|_\S$ and $\SFF$ satisfy the Gauss-Codazzi-Mainardi equations \cite{DOC76}. If $\go_\S$ and $\SFF$ are incompatible, then the plates/shell is residually-stressed and the minimizer of \eqref{eq:plates} has non-zero energy.

\subsection{Rods}

Rods correspond to the case $n=3$ and $k=2$, which means that $\S$ is a one-dimensional submanifold. For $n-k=1$ it is $dF$ that is uniquely determined by $\qperp$.

The remarkable fact in the case of a one-dimensional submanifold is that there exists a limiting configuration, $(F,\qperp)$, such that the limiting energy $\Elim$ vanishes, That is, $dF\oplus\qperp\in\SO{n}$ and
\beq
\Pperpq\circ\nabla\qperp=0
\Textand
\Pparq\circ\nabla\qperp=\SFF.
\label{eq:rods}
\eeq
This means that there exists a sequence of approximate minimizers $f_h$ for which $\Eh[f_h]=o(1)$. Hence, to obtain a finer limiting structure one should divide the physical elastic energy by a higher power of $h$.

We now show that a solution satisfying \eqref{eq:rods} does exist. In particular, we show that equations~\eqref{eq:rods} are equivalent to a system of linear ordinary differential equations. Let $p\in\S$.
Since $\qpar\circ\Pparq = \id - \qperp\circ\Pperpq$, it follows that
\[
\qpar\circ \Pparq\circ\nabla\qperp = \nabla\qperp - \qperp\circ\Pperpq \circ\nabla\qperp.
\]
Substituting using equations~\eqref{eq:rods}, we obtain
\beq
\nabla\qperp = \qpar\circ \SFF.
\label{eq:serret-frenet}
\eeq
On the other hand, equation~\eqref{eq:serret-frenet} immediately implies equations~\eqref{eq:rods}.

Moreover, an equation for $\qpar$ can be derived from~\eqref{eq:serret-frenet} as follows. Since for every $\xi\in\NS$ and $\eta\in T\S$
\[
\euc(\qperp(\xi),\qpar(\eta)) = 0,
\]
differentiating and using equation~\eqref{eq:serret-frenet}, we obtain
\[
\begin{split}
\euc(\qperp(\xi),\nabla_{\eta'} \qpar(\eta)) &= - \euc(\nabla_{\eta'}\qperp(\xi),\qpar(\eta)) \\
&=
- \euc(\qpar\circ\SFF(\eta',\xi),\qpar(\eta)) = -\go(\SFF(\eta',\xi),\eta).
\end{split}
\]
Introducing the metric adjoint of the second fundamental form $\hat{\SFF}:T\S\times T\S\to\NS$, we have
\[
\euc(\qperp(\xi),\nabla_{\eta'} \qpar(\eta)) = -\go(\xi,\SFF(\eta,\eta')) = - \euc(\qperp(\xi),\qperp\circ\SFF(\eta,\eta')),
\]
from which we conclude that
\beq
\nabla\qpar = -\qperp\circ \hat{\SFF}.
\label{eq:serret-frenet2}
\eeq
Since $\S$ is one-dimensional, \eqref{eq:serret-frenet} and \eqref{eq:serret-frenet2} form a  linear ordinary differential system, which can be solved globally.  In fact, they are generalized Serret-Frenet equations, which uniquely determine the shape of the limiting curve.

\section{Discussion}

The main contribution of this paper is to unify the analyses of dimensionally-reduced elasticity models, resulting in a limiting model that covers a variety of known cases as well as cases that have not yet been rigorously treated, such as non-Euclidean shells and rods. The generalization has been attained by performing the entire analysis within the framework of Riemannian geometry. While being at times technically cumbersome, the Riemannian formalism is the appropriate framework when considering incompatible elastic bodies, so that the notion of a Euclidean reference configuration becomes irrelevant.

The entire analysis rests on the so-called finite bending assumption, whereby there exists a family of mappings $f_h$, such that $\Eh[f_h] = O(1)$. As noted by Lewicka and Pakzad \cite{LP10}, this condition is equivalent to the existence of a $W^{2,2}$ isometric immersion of $(\S,\go|_\S)$ in $\R^n$. The equivalence between the two conditions in an immediate corollary of our results: if the finite bending assumption holds then we have shown that $f_h$ has a subsequence that \rr-converges to a $W^{2.2}(\S;\R^n)$ isometric immersion. Conversely, if a $W^{2,2}$ isometric immersion exists, then the recovery sequence defined in Section~\ref{sec:recovery} satisfies the finite bending assumption.

It should be noted that in the present work we assumed a specific elastic energy functional \eqref{eq:Eh}. Note that this energy is tolerant to local orientation reversal, which is clearly unphysical. A more general treatment would assume, as customary,  an arbitrary energy density along with the standard Lipschitz continuity and coercivity conditions.  We intentionally chose a specific energy in order to avoid additional technicalities; the analysis in \cite{FJM02b}, for example, indicates that  a more general energy can be addressed by the same techniques.

This work raises a number of questions of interest.
\begin{itemize}
\item
Under what conditions on $\go|_\S$ does the finite bending assumption hold? There exists a large amount of work on H\"older regular or smooth isometric immersions in Euclidean space. However, we are not aware of similar work for Sobolev maps.

\item
The proof of local rigidity, Theorem~\ref{th:rig1}, relies on the rigidity theorem proved in \cite{FJM02b} for mappings $\R^n\to\R^n$. Generalizations to mappings between Riemannian manifolds would be of much interest.

\item
How does $f_h$ approach the limit $(F,\qperp)$, and in particular, how does $\Eh[f_h] - \Elim[F,\qperp]$ scale with $h$. The (non-rigorous) analysis in \cite{ESK09} indicates that the deviation of $f_h$ from the limit $F$ is focused in a boundary layer of width $O(h^{1/2})$ in the vicinity of the boundary of $\S$, and consequently, $\Eh[f_h] - \Elim[F,\qperp] = O(h^{1/2})$.

\end{itemize}

\paragraph{Acknowledgements}
Acknowledgments: We are grateful to Hillel Aharoni and Michael Moshe for useful discussions.
RK was partially supported by the Israeli Science Foundation as well as by the Binational Israel-US Science Foundation. JS was partially supported by the Israel Science Foundation grant 1321/2009 and the Marie Curie International Reintegration Grant No. 239381.


\bibliographystyle{unsrt}
\bibliography{MyBibs}

\end{document}